\documentclass[reqno]{amsart}
\usepackage[margin=1in]{geometry}
\usepackage{graphicx}
\usepackage{physics}
\usepackage{braket}
\usepackage{enumitem}
\usepackage{amsthm}
\usepackage{amssymb}
\usepackage{amsmath}
\usepackage{mathtools}
\usepackage{mathrsfs}
\usepackage[dvipsnames]{xcolor}
\usepackage{wasysym}
\usepackage{tikz-cd}
\usepackage[nocompress]{cite}

\usepackage{hyperref}
\hypersetup{
    hyperindex=true,
    bookmarks=true,
    pdfa,
    colorlinks=true,
    breaklinks=true,
    urlcolor=red,
    linkcolor=red,
    citecolor=red,
    bookmarksopen=true,
}

\renewcommand{\ip}[1]{\langle #1 \rangle}
\newcommand{\too}{\longrightarrow}
\newcommand{\mtoo}{\longmapsto}
\newcommand{\mto}{\mapsto}
\newcommand{\C}{\mathbb{C}}
\newcommand{\R}{\mathbb{R}}
\newcommand{\Z}{\mathbb{Z}}

\renewcommand{\S}{\mathcal{S}}
\renewcommand{\H}{\mathcal{H}}
\newcommand{\B}{\mathcal{B}}
\newcommand{\E}{\mathcal{E}}
\renewcommand{\P}{\mathcal{P}}
\newcommand{\GL}{\operatorname{GL}}

\newcommand{\U}{\operatorname{U}}
\newcommand{\G}{\mathcal{G}}

\newcommand{\Aut}{\operatorname{Aut}}

\newcommand{\defn}[1]{\textbf{\textit{#1}}}

\newcommand{\Mp}{\mathrm{Mp}}
\newcommand{\Sp}{\mathrm{Sp}}

\newcommand{\im}{\operatorname{im}}
\newcommand{\HH}{\mathbb{H}}

\newcommand{\e}{\varepsilon}
\newcommand{\vol}{\operatorname{vol}}
\newcommand{\longhook}{\lhook\joinrel\longrightarrow}
\newcommand{\diag}{\operatorname{diag}}
\newcommand{\Heis}{\mathrm{Heis}}
\newcommand{\Pauli}{\mathrm{Pauli}}
\newcommand{\Cliff}{\mathrm{Cliff}}
\newcommand{\code}{\mathscr{C}}
\newcommand{\passive}{\mathrm{passive}}

\newcommand{\FDA}{{\mathcal{V}_{\Lambda^\perp}}}
\newcommand{\FDB}{\mathcal{V}_\Lambda}

\numberwithin{equation}{section}

\theoremstyle{plain}
\newtheorem{theorem}{Theorem}[section]
\newtheorem{proposition}[theorem]{Proposition}
\newtheorem{lemma}[theorem]{Lemma}
\newtheorem{corollary}[theorem]{Corollary}

\newtheorem{problem}[theorem]{Problem}
\theoremstyle{definition}
\newtheorem{definition}[theorem]{Definition}
\newtheorem{remark}[theorem]{Remark}
\newtheorem{example}[theorem]{Example}

\title[Complex abelian varieties and quantum error correction]{Complex abelian varieties and quantum error correction:\\ a mathematical framework for GKP codes}
\author{Maxence Mayrand and Baptiste Royer}
\address[Maxence Mayrand]{D\'{e}partement de math\'{e}matiques et Institut Quantique \\ Universit\'{e} de Sherbrooke \\ 2500 Bd de l'Universit\'{e} \\ Sherbrooke, QC, J1K 2R1, Canada}
\email{maxence.mayrand@usherbrooke.ca}
\address[Baptiste Royer]{D\'{e}partement de physique et Institut Quantique\\ Universit\'{e} de Sherbrooke \\ 2500 Bd de l'Universit\'{e} \\ Sherbrooke, QC, J1K 2R1, Canada}
\email{baptiste.royer@usherbrooke.ca}

\begin{document}

\begin{abstract}
We study a class of quantum error-correcting codes through the geometry of complex abelian varieties. These codes, introduced by Gottesman--Kitaev--Preskill, are built from symplectically integral lattices and therefore naturally define polarized complex abelian varieties. We give a precise mathematical formulation of this relationship and extend it to a dictionary between the main structures of GKP code theory and classical objects in the theory of abelian varieties. For instance, under this dictionary, the finite-dimensional code space becomes the space of theta functions $H^0(X, L)$, logical Pauli gates arise from the theta group, passive logical Clifford gates correspond to automorphisms of the polarized abelian variety, and concatenation with stabilizer codes corresponds to isogeny. We also prove several key results that give precise mathematical formulations of statements about these codes that often appear in heuristic form in the physics literature. In particular, we prove that the encoding is asymptotically isometric, that every logical Clifford gate is realized by a Gaussian unitary, and that, for noise of small variance, the failure probability is governed to first order by the shortest nontrivial displacement in the kernel of the polarization isogeny, a systolic invariant of the underlying polarization. This leads naturally to optimization problems on the moduli space of polarized abelian varieties.
\end{abstract}

\maketitle

\section{Introduction}

\subsection{Overview}

Quantum computation studies finite-dimensional Hilbert spaces of the form
\[
\C^2 \otimes \cdots \otimes \C^2
\]
or more generally $\C^{d_1} \otimes \cdots \otimes \C^{d_n}$ with $d_i \ge 2$, whose nonzero vectors represent quantum information and whose unitary operators represent quantum gates. A quantum algorithm consists of a sequence of such gates and measurements applied to a given initial state.

A central challenge is to realize the abstract system $\C^2 \otimes \cdots \otimes \C^2$ and its operations within a physical system in a manner that is robust against background noise. Since noise is fundamentally unavoidable, the problem is not to eliminate it, but to detect and correct it in a manner that does not disturb the abstract system. This is the subject of quantum error correction.

One particularly useful approach to quantum error correction is provided by the construction of Gottesman--Kitaev--Preskill~\cite{GKP}, known as GKP codes. These codes have attracted considerable attention in recent years as a promising approach toward quantum computation in various physical platforms~\cite{matsos2025universal,campagne2020quantum,sivak2023real,deneeveErrorCorrectionLogical2022,lachance-quirionAutonomousQuantumError2024,lemondeHardwareEfficientFaultTolerant2024,bourassaBlueprintScalablePhotonic2021,larsenIntegratedPhotonicSource2025,aghaeeradScalingNetworkingModular2025}.

The first goal of this paper is to give a mathematically rigorous foundation for GKP codes in the language of lattices, Heisenberg groups, and their representations. This refines and places on firm mathematical footing earlier descriptions appearing in the physics literature \cite{HarringtonPreskill2001,royer2022encoding,conrad2022gottesman}. In particular, we prove several key results that give precise mathematical formulations of statements about these codes that often appear only in heuristic form. Altogether, this provides a conceptual foundation for the theory and illuminates several of its features.

The second goal of this paper is to develop the relationship between GKP codes and complex abelian varieties. This relationship was already noted in early work of Harrington--Preskill~\cite{HarringtonPreskill2001} and appears in some recent works on GKP codes \cite{conrad2024lattices,burchards2025fiber,conrad2024fabulous}. These works observe that the lattice underlying a GKP code naturally defines a complex abelian variety, and they explore some consequences of this perspective. Our aim is to build on this connection and formulate it systematically and rigorously from the point of view of complex geometry, not merely as a formal analogy, but as a broader correspondence in which many of the key ingredients of GKP codes acquire natural geometric interpretations.

The resulting picture is that many central notions from quantum error correction become familiar objects from the theory of abelian varieties, while the error-correction aspects of the theory introduce new geometric questions. For instance, automorphisms of a polarized complex abelian variety correspond to logical quantum gates, whereas robustness leads naturally to the study of probability measures on abelian varieties and their interaction with the underlying geometry. More concretely, for Gaussian displacement noise, robustness becomes a function on the moduli space of polarized abelian varieties of fixed type, and its small-noise asymptotics are governed by a natural systolic function on this moduli space.

This geometric reformulation provides a bridge between quantum error correction and complex geometry. On the one hand, it places GKP codes in a well-established mathematical setting, where classical tools from the theory of abelian varieties can be used to study questions of quantum error correction. On the other hand, it gives a new interpretation of complex abelian varieties as geometric models for quantum error-correcting codes, suggesting fresh questions and problems in complex geometry.

Thus the paper has two complementary aims: to make the mathematical structure of GKP codes precise, and to reinterpret this structure through the theory of complex abelian varieties, thereby creating a common language between quantum error correction and complex geometry.

\subsection{Summary of results}

We now give a detailed outline of the main definitions and results of the paper from the point of view of complex geometry, together with some of the problems they suggest.

Let $(X, L)$ be a complex abelian variety $X$ together with an ample line bundle $L \to X$. Let $(d_1, \ldots, d_n)$ be the type of $(X, L)$ and set
\[
D \coloneqq \diag(d_1, \ldots, d_n).
\]
The space $H^0(X, L)$ of holomorphic sections of $L$ has dimension $d_1\cdots d_n$, the same as the standard qudit Hilbert space
\begin{equation}\label{u3l7btgj}
\C^{d_1} \otimes \cdots \otimes \C^{d_n}
\end{equation}
from quantum information theory.
These two spaces can in fact be naturally identified by a choice of \emph{theta structure} on $(X, L)$. Indeed, recall that the theta group of $L$ is a Heisenberg group acting irreducibly on $H^0(X, L)$, and a theta structure is an isomorphism between this theta group and the standard finite Heisenberg group of type $D$. By the Stone--von Neumann theorem, such a theta structure determines an identification
\[
\C^{d_1} \otimes \cdots \otimes \C^{d_n} \cong H^0(X, L),
\]
unique up to a global phase.

Now, by realizing $H^0(X, L)$ as the space of canonical theta functions and using the Bargmann transform to identify them with tempered distributions on $\R^n$, we get a linear embedding
\begin{equation}\label{l4rwp0gw}
\mathrm{Enc} : \C^{d_1} \otimes \cdots \otimes \C^{d_n} \cong H^0(X, L) \longhook \mathcal{S}'(\R^n),
\end{equation}
called an \emph{encoding}.
The target $\mathcal{S}'(\R^n)$ represents an ideal physical system, whose states include square-integrable functions, i.e.\ wave functions, as well as idealized states such as Dirac distributions. From a physics point of view, $\mathcal{S}'(\R^n)$ models a harmonic oscillator with $n$ modes. The domain and target of such an encoding are called the \emph{logical} and \emph{physical} systems, respectively, denoted
\begin{equation}\label{w2x7pkag}
\mathscr{H}_{\mathrm{logical}} \coloneqq \C^{d_1} \otimes \cdots \otimes \C^{d_n}
\end{equation}
and
\begin{equation}\label{pqk64lha}
\mathscr{H}_{\mathrm{physical}} \coloneqq \mathcal{S}'(\R^n).
\end{equation}
In this way, every polarized complex abelian variety of type $(d_1, \ldots, d_n)$ together with a theta structure determines an encoding of the logical system \eqref{w2x7pkag} into the ideal physical system \eqref{pqk64lha}. This is the geometric incarnation of a GKP code. In Section \ref{ltg2p7xl}, we formulate a precise mathematical definition of GKP codes in the more standard language of lattices and translation operators, and prove that this definition is equivalent to the construction above. Thus GKP codes are in one-to-one correspondence with polarized complex abelian varieties endowed with a theta structure.

To obtain an actual quantum error correction scheme, we need three more properties, which can be informally stated as follows. More precise statements will appear in the body of the paper.
\begin{enumerate}[label={\textup{(\arabic*)}}]
    \item \label{xmymvki3} The standard hermitian inner product on $\mathscr{H}_{\mathrm{logical}}$ can be recovered from the geometry of $\mathscr{H}_{\mathrm{physical}}$.
    \item \label{afegowxw} Many logical gates of interest, i.e.\ unitary operators on $\mathscr{H}_{\mathrm{logical}}$, can be implemented by physically realistic unitary operators on $\mathscr{H}_{\mathrm{physical}}$.
    \item \label{gacgo30t} If a sufficiently small displacement error sends a state $\psi \in \im \mathrm{Enc}$ outside the image of the encoding, then the error can be detected by measurement and corrected by a unitary operator.
\end{enumerate}

One of the main goals of the paper is to give precise mathematical statements and proofs for these three properties. We outline them in turn.

\subsubsection{Asymptotic isometry of the encoding}

The first property concerns the inner product. In standard quantum error correction, one usually requires the embedding $\mathscr{H}_{\mathrm{logical}} \hookrightarrow \mathscr{H}_{\mathrm{physical}}$ to be an isometry. For GKP codes, this requirement cannot hold directly in the usual sense, since $\mathscr{H}_{\mathrm{physical}} = \mathcal{S}'(\R^n)$ is not a Hilbert space. Indeed, the encoded states are not, in general, square-integrable, so the usual $L^2$ inner product cannot be applied to them directly. Nevertheless, the ideal encoded states admit natural square-integrable approximations.

\begin{theorem}[Theorem \ref{hr1j1x4z}]
There exists a canonical family of encodings
\[
\mathrm{Enc}_\beta : \C^{d_1} \otimes \cdots \otimes \C^{d_n} \too L^2(\R^n),
\qquad \beta > 0,
\]
which converge to the ideal encoding \eqref{l4rwp0gw} as $\beta \to 0$ and are asymptotically isometric up to a conformal factor. In particular, there exists $a(\beta) > 0$ such that
\[
\ip{\varphi, \psi} =
\lim_{\beta \to 0}
a(\beta)
\ip{\mathrm{Enc}_\beta(\varphi), \mathrm{Enc}_\beta(\psi)}_{L^2(\R^n)}
\]
for all $\varphi, \psi \in \C^{d_1} \otimes \cdots \otimes \C^{d_n}$.
\end{theorem}

Thus, although the ideal encoding naturally lands in $\mathcal{S}'(\R^n)$, the hermitian inner product on the logical system can be recovered from square-integrable approximations. This establishes Property \ref{xmymvki3}.

\subsubsection{Logical gates and automorphisms}

The second property concerns logical gates, i.e.\ unitary operators on
$\mathscr{H}_{\mathrm{logical}}$. Through the theta structure, we identify
\[
\mathscr{H}_{\mathrm{logical}} \cong H^0(X, L),
\]
so the problem is to understand which unitary operators on $H^0(X, L)$ can be implemented by natural operators on $\mathscr{H}_{\mathrm{physical}}$.

A first important class of logical gates is given by the Pauli and Clifford gates. In the usual qudit formalism, these groups are described using the tensor product decomposition
\[
\C^{d_1} \otimes \cdots \otimes \C^{d_n}
\]
and a choice of standard coordinates. However, this description depends on noncanonical choices. In Section \ref{tj099yyk}, we develop a basis-independent formalism for Pauli and Clifford groups associated with finite symplectic abelian groups. This formalism is well suited to GKP codes and theta groups, and may also be of independent interest. More precisely, starting from a finite symplectic abelian group $(K,\omega)$, one considers its Heisenberg group and the corresponding Stone--von Neumann representation. The Pauli group is then a canonical finite subgroup of this Heisenberg group, and the Clifford group is the normalizer of its image in the unitary group of the representation.

In the present geometric setting, the finite symplectic abelian group is
\[
K(L) = \{x \in X : t_x^*L \cong L\},
\]
where $t_x : X \to X$ are the translations, and its Heisenberg group is the compact theta group $\G_c(L)$. Thus the action of $\G_c(L)$ on $H^0(X, L)$ is the Stone--von Neumann representation associated with $K(L)$. Let $m = d_n$ be the largest elementary divisor of the type $(d_1, \ldots, d_n)$. Using the formalism of Section \ref{tj099yyk}, we show that the standard Pauli group $\Pauli_D$ on $\C^{d_1} \otimes \cdots \otimes \C^{d_n}$ corresponds to the subgroup
\[
\P(L) \coloneqq \{g \in \G_c(L) : g^{2m} = 1\}.
\]
The Clifford group $\Cliff_D$ on $\C^{d_1} \otimes \cdots \otimes \C^{d_n}$ then corresponds to the normalizer of the image of this Pauli group in $\U(H^0(X,L))$.

We next show that Clifford gates can be realized by natural unitary operators on $\mathscr{H}_{\mathrm{physical}}$. The relevant operators come from the extended metaplectic representation
\begin{equation}\label{ajjzgq1y}
\rho : \Mp_{2n} \ltimes \HH_n \too \GL(\mathcal{S}'(\R^n)),
\end{equation}
where $\Mp_{2n}$ is the metaplectic group and $\HH_n$ is the real Heisenberg group. These operators preserve $L^2(\R^n)$ and restrict to unitary operators there. In the physics literature, they are known as \emph{Gaussian unitaries} \cite{weedbrook2012gaussian}.

\begin{theorem}[Theorem \ref{cczuu62j}]
There is a canonical subgroup $N \subset \Mp_{2n} \ltimes \HH_n$ such that $\rho|_N$ preserves the image of the encoding \eqref{l4rwp0gw}. Moreover, the induced representation
\[
N \too \GL(\C^{d_1} \otimes \cdots \otimes \C^{d_n})
\]
is unitary and its image is precisely the group $\Cliff_D$ of Clifford gates.
\end{theorem}

Thus every Clifford gate can be implemented by an operator coming from the extended metaplectic representation. This establishes Property \ref{afegowxw} in a strong form.

Among these operators, an especially important role is played by the so-called \emph{passive} ones, namely those arising from the subgroup $\U(n) \ltimes \HH_n$, where $\U(n) = \Sp_{2n} \cap \mathrm{O}(2n)$ embeds in $\Mp_{2n}$ in the usual way.\footnote{\label{f3rsimhu}This differs slightly from standard terminology in the physics literature, where only those with trivial $\HH_n$ component are called passive. We adopt the broader convention here because it is more convenient for the purposes of this paper.} These operators are much easier to realize experimentally, so it is natural to ask which Clifford gates they induce. Since Pauli gates and phases arise from this subgroup, we want the quotient group
\[
\Cliff_D^\passive / (\U(1)\Pauli_D)
\]
to be as large as possible.

This question has a natural geometric answer. Let $\Aut(X, L)$ be the group of biholomorphisms $f : X \to X$ fixing the origin such that $c_1(f^*L) = c_1(L)$. Then $\Aut(X, L)$ preserves $K(L)$ and we denote the pointwise stabilizer by
\[
\Aut(X, L)_{K(L)} \coloneqq \{ f \in \Aut(X, L) : f(x) = x \text{ for all } x \in K(L)\}.
\]

\begin{theorem}[Theorem \ref{fi2al08e}]
There is a short exact sequence
\[
1 \too \Aut(X, L)_{K(L)} \too \Aut(X, L) \too \Cliff_D^\passive/(\U(1)\Pauli_D) \too 1.
\]
\end{theorem}

Equivalently, $\Cliff_D^\passive/(\U(1)\Pauli_D)$ is isomorphic to the image of the natural map $\Aut(X, L) \to \Sp(K(L))$, where $\Sp(K(L))$ is the group of automorphisms of $K(L)$ preserving the commutator map $e^L$.

Thus, for a fixed type, finding GKP codes with many non-Pauli passive Clifford
gates amounts to finding polarized complex abelian varieties whose automorphism
groups have large image in $\Sp(K(L))$.

\begin{problem}\label{pam4y8n6}
For a given type $D$, find a polarized complex abelian variety $(X, L)$ of this type such that the image of $\Aut(X, L)$ in $\Sp(K(L))$ is as large as possible.
\end{problem}

Such a polarized complex abelian variety produces a GKP code with many logical gates that are easy to implement physically. The problem of finding polarized abelian varieties with large automorphism groups has been well studied; see, for example, \cite{gonzalez2009automorphisms} and the references therein. Although one must also understand the image of these automorphism groups in $\Sp(K(L))$, complex abelian varieties of CM type provide a potentially useful source of examples.

Via the Torelli theorem, another natural source of examples is provided by Riemann surfaces with large automorphism groups (Corollary \ref{zdzp6h79}). Their Jacobians are principally polarized complex abelian varieties of type $(1, \ldots, 1)$, which can be rescaled to polarization type $(d, \ldots, d)$. This perspective leads naturally to the study of GKP codes associated with Hurwitz curves, that is, Riemann surfaces whose automorphism groups attain the Hurwitz bound $84(g-1)$, where $g$ is the genus. In particular, the Klein quartic emerges as a promising candidate, yielding, for instance, a GKP code for three qubits $\C^2 \otimes \C^2 \otimes \C^2$ with $168$ passive Clifford gates, modulo Pauli gates and phases; see Proposition \ref{dhnyif6v}. We will analyze the physical implications of this construction in future work.

\subsubsection{Decoding, robustness, and systoles}\label{fci02gnv}

We now discuss Property \ref{gacgo30t}, i.e.\ error correction. This is where a genuinely new ingredient enters the geometry of abelian varieties: a probability measure $\rho$ on $X$. From the point of view of quantum error correction, this measure describes the distribution of possible errors, and robustness becomes a question about how the geometry of $(X,L)$ interacts with $\rho$.

To define this interaction, we first explain how an individual error appears in the geometric picture. A displacement error is represented geometrically by a point $x \in X$. If a code state is represented by a section
\[
\psi \in H^0(X, L),
\]
then the error $x$ changes it into a section
\begin{equation}\label{tgdyegy2}
t_x^*\psi \in H^0(X, t_x^*L)
\end{equation}
of the translated line bundle.
Thus the error does not merely move the section; it also moves the line bundle. The family of all such possible spaces is organized by the vector bundle
\begin{equation}\label{8ae63z8g}
\mathscr{M} : \bigsqcup_{y \in \widehat{X}} H^0(X, y \otimes L) \too \widehat{X},
\end{equation}
where $\widehat{X} = \mathrm{Pic}^0(X)$ is the dual abelian variety and we identify a point $y \in \widehat{X}$ with the corresponding degree-zero line bundle.

One would like to detect and correct such an error. However, in the physical setting, the error $x$ itself cannot be directly observed. Instead, the detectable information is the degree-zero line bundle by which $L$ has been translated. In other words, one could make a measurement on a state $\phi$ in the total space of \eqref{8ae63z8g} and recover the base point $\mathscr{M}(\phi) \in \widehat{X}$, while leaving the state $\phi$ unchanged. In practice, this information is obtained by measuring certain translation operators on $\mathcal{S}'(\R^n)$; see Section \ref{ltg2p7xl}. More precisely, let
\[
\phi_L : X \too \widehat{X}
\]
be the polarization isogeny. If the error is $x$, then the detectable datum is
\[
\phi_L(x) \in \widehat{X}.
\]
In the language of quantum error correction, this datum is called the \emph{syndrome}. Thus the syndrome does not determine $x$ itself, but only the finite fibre $\phi_L^{-1}(\phi_L(x))$ containing it. The goal of decoding is then to choose a representative $z \in \phi_L^{-1}(\phi_L(x))$ and apply the inverse displacement to $t_x^*\psi$ in an attempt to recover $\psi$.

A \emph{decoder} is therefore a measurable section
\[
\mathcal{D} : \widehat{X} \too X
\]
of $\phi_L$. Given a measured syndrome $y \in \widehat{X}$, the decoder chooses the representative $\mathcal{D}(y) \in \phi_L^{-1}(y)$ and applies the inverse displacement. The decoder succeeds on an error $x \in X$ precisely when
\[
\mathcal{D}(\phi_L(x)) = x,
\]
or equivalently when $x \in \im \mathcal{D}$. Hence,
\[
\text{(probability of success using the decoder $\mathcal{D}$)}
=
\rho(\im \mathcal{D}).
\]
The \emph{robustness} of the code is the largest possible success probability,
\[
\mathscr{R}_\rho(X, L)
=
\sup_{\mathcal{D}} \rho(\im \mathcal{D}),
\]
where the supremum is taken over all measurable sections of $\phi_L$. Thus robustness is a number in $[0,1]$ that measures how well the geometry of $(X,L)$ is adapted to the probability measure $\rho$. For example, if $\rho$ is the Haar measure, then $\mathscr{R}_\rho(X, L) = 1/|K(L)|$ (Example \ref{tenixxix}). In general, we have $\mathscr{R}_\rho(X, L) \ge 1/|K(L)|$ (Proposition \ref{3d9mrs7s}). The robustness $\mathscr{R}_{\rho}(X, L)$ depends only on $X$ and the polarization $c_1(L)$, but not on the specific choice of ample line bundle $L$. In the cases considered in this paper, this supremum is attained by an explicit decoder; see Proposition \ref{rq2fhiu0}.

One important case for quantum error correction is Gaussian displacement noise~\cite{GKP,noh2018quantum,hastrup2021analysis}. Let $\pi : V \to X$ be the universal cover of $X$, equipped with the hermitian metric induced by the polarization. For each $\sigma>0$, let $\rho_\sigma$ be the pushforward to $X$ of the centered Gaussian measure of variance $\sigma^2$ on $V$ with respect to this metric. Since this measure is defined only in terms of the polarization metric, it is intrinsic to the polarized abelian variety. Therefore, for each polarization type $D$ and each $\sigma > 0$, robustness defines a function
\begin{equation}\label{jcaskgly}
\mathscr{R}_{\sigma} : \mathcal{A}_D \too [0, 1],
\end{equation}
on the moduli space $\mathcal{A}_D$ of polarized complex abelian varieties of type $D$.

This leads to the following natural optimization problem.

\begin{problem}\label{86b8npxy}
For each polarization type $D$ and each $\sigma>0$, study the robustness function \eqref{jcaskgly}. In particular, determine its supremum and describe the polarized abelian varieties, if any, for which this supremum is attained.
\end{problem}

More generally, one may consider other intrinsic noise models on polarized abelian varieties, provided the corresponding probability measures are defined in a way that is compatible with isomorphisms of polarized abelian varieties.

For the Gaussian measure $\rho_\sigma$, one is especially interested in the small-variance regime, which represents small displacement errors. In this regime, the failure probability is controlled by a simple geometric invariant of the polarized abelian variety. Let $\ell_{X,L}$ be the length of the shortest geodesic segment joining the origin of $X$ to a nonzero point of $K(L)$ and let $N_{X,L}$ be the number of points of $K(L)$ at this minimal distance.

\begin{theorem}[Theorem \ref{clp62a30}]
Let $(X, L)$ be a polarized complex abelian variety. Then
\[
1 - \mathscr{R}_{\sigma}(X, L) \sim \frac{2N_{X,L} \sigma}{\ell_{X,L}\sqrt{2\pi}} \exp(-\frac{\ell_{X,L}^2}{8\sigma^2}) \quad \text{as } \sigma \to 0.
\]
\end{theorem}

Since $\ell_{X, L}$ is an invariant of the polarized complex abelian variety, it defines a function
\begin{equation}\label{gyhlyulh}
\ell : \mathcal{A}_D \too \R_{>0}.
\end{equation}
The small-noise limit therefore reduces the robustness optimization problem to a systolic optimization problem on $\mathcal{A}_D$.

\begin{problem}\label{j66sypn9}
For each polarization type $D$, study the systolic function \eqref{gyhlyulh}. In particular, determine its supremum and characterize the polarized abelian varieties, if any, for which this supremum is attained.
\end{problem}

In particular, if $(X, L)$ is a principally polarized complex abelian variety, then $(X, L^d)$ has type $(d, \ldots, d)$ and
\[
\ell_{X,L^d} = \frac{1}{d}\lambda_1(X, L),
\]
where $\lambda_1(X, L)$ is the systole of $X$, i.e.\ the length of the shortest closed geodesic, or equivalently twice the sphere packing radius of the corresponding lattice. Thus, in type $(d, \ldots, d)$, Problem \ref{j66sypn9} specializes to the classical problem of maximizing the systole of a polarized abelian variety, sometimes referred to as the \emph{Buser--Sarnak invariant}, following \cite{BuserSarnak}. This special case was studied, in particular, in \cite{Lazarsfeld1996,BuserSarnak,Bauer1998}. 
In dimensions $1$ and $2$, Buser--Sarnak \cite{BuserSarnak} identify the maximizers: the elliptic curve associated with the hexagonal lattice, and the abelian surface associated with the $D_4$ lattice. In dimension $3$, they conjecture that the supremum is attained by the Jacobian of the Klein quartic.

\subsubsection{Concatenation and isogeny}

The final construction we discuss is concatenation. In quantum error correction, concatenation means encoding the logical space of one code into another code. In the setting of GKP codes, this operation has a simple geometric interpretation.

Let $(X = V/\Lambda, L)$ be the polarized complex abelian variety associated with a GKP code. Concatenating this code with a finite stabilizer code amounts, on the lattice side, to choosing an intermediate lattice
\[
\Lambda \subset \widetilde{\Lambda} \subset \Lambda^\perp,
\]
together with the appropriate compatibility condition on the semicharacter. Geometrically, this is the same as passing to the quotient
\[
\widetilde{X} = V/\widetilde{\Lambda},
\]
and hence to an isogeny
\[
X = V/\Lambda \too \widetilde{X} = V/\widetilde{\Lambda}.
\]
In Section \ref{cu43tyv7}, we make this precise and show that concatenation of GKP codes with stabilizer codes is equivalent to passing to suitable isogenies of polarized complex abelian varieties.

\subsubsection{The dictionary}

The correspondence developed in this paper can be summarized as follows:
\[
\begin{array}{c|c}
\text{quantum error correction} & \text{complex abelian varieties} \\
\hline
\text{GKP code} & \text{polarized complex abelian variety }(X, L) \\
\text{code space} & H^0(X, L) \\
\text{choice of encoding} & \text{theta structure} \\
\text{logical Pauli gates} & \text{finite subgroup of the theta group} \\
\text{passive logical Clifford gates} & \text{image of } \Aut(X, L) \text{ in } \Sp(K(L))\\
\text{decoder} & \text{measurable section of }\phi_L : X \to \widehat{X} \\
\text{robustness} & \text{maximal mass of a section of }\phi_L \\
\text{small-noise failure probability} & \text{systolic invariant} \\
\text{concatenation} & \text{isogeny}
\end{array}
\]

\subsection{Organization of the paper}

Section \ref{l9qcvbkk} recalls the necessary background on Heisenberg groups, lattices, and complex abelian varieties. Section \ref{tj099yyk} develops a basis-independent formalism for Pauli and Clifford groups associated with finite symplectic abelian groups. Section \ref{ltg2p7xl} gives a precise mathematical definition of GKP codes and proves their correspondence with polarized complex abelian varieties endowed with theta structures.

The three properties discussed above are then treated in turn. Section \ref{b9j929je} proves the asymptotic isometry of the encoding. Section \ref{5nk36oj5} studies logical gates and their realization by Gaussian unitaries, including the geometric interpretation of passive Clifford gates in terms of automorphisms of polarized abelian varieties. Section \ref{k3dqdqez} develops the theory of decoding and robustness, and proves the small-noise asymptotic formula. Finally, Section \ref{cu43tyv7} explains the relationship between concatenation and isogenies.

\subsection{Acknowledgments}
The authors acknowledge support from the New Frontiers in Research Fund (Exploration stream), Application ID: NFRFE-2025-00668. M.M.\ acknowledges support from a Discovery Grant (RGPIN-2023-04587) from the Natural Sciences and Engineering Research Council of Canada (NSERC). B.R.\ acknowledges support from an NSERC Discovery Grant (RGPIN-2022-04451) and Fonds de Recherche du Québec - Nature et Technologies.

\subsection{Notation and conventions}

We use the convention that hermitian inner products are linear in the first variable and conjugate-linear in the second. For a lattice $\Lambda$ in a real symplectic vector space $(V,E)$, its symplectic dual is
\[
\Lambda^\perp \coloneqq \{\mu \in V : E(\mu,\Lambda)\subset \Z\}.
\]
When a euclidean or hermitian metric is fixed, $\det(\Lambda)$ denotes the covolume of $\Lambda$, and $\Lambda^*$ denotes the euclidean dual lattice with respect to this metric. If $V$ is viewed as a hermitian vector space with hermitian form $H$, we write $|v|^2 = H(v,v)$.

For functions $u(t)$ and $v(t)$, we use the standard notation
\[
u(t)=O(v(t)) \quad \text{as } t \to t_0
\]
to mean that there exist constants $C>0$ and $\e > 0$ such that
\[
|u(t)|\le C|v(t)|
\]
whenever $0<|t-t_0| < \e$.
We write
\[
u(t)=o(v(t)) \quad \text{as } t \to t_0
\]
to mean that
\[
\lim_{t\to t_0}\frac{u(t)}{v(t)}=0.
\]
For positive functions $u(t)$ and $v(t)$, we write
\[
u(t)\sim v(t) \quad \text{as } t \to t_0
\]
to mean that
\[
\lim_{t\to t_0}\frac{u(t)}{v(t)}=1.
\]

\section{Mathematical preliminaries: Heisenberg groups, lattices, and complex abelian varieties}\label{l9qcvbkk}

This section recalls the background material needed throughout the paper and fixes the normalizations used later. We begin with symplectic abelian groups and Heisenberg groups, then recall the Schr\"odinger and Fock realizations of the real Heisenberg representation and the Bargmann transform relating them. We then turn to symplectically integral lattices and explain how they give rise to complex abelian varieties, theta functions, and theta groups. Most of the material is standard, but we spell out the normalizations carefully, since the precise factors of $2\pi$, the choice of hermitian form, and the comparison between the Schr\"odinger and Fock realizations will be important in later sections. Standard references include \cite{Mumford-Tata-I,Mumford-Tata-II,Mumford-Tata-III,Polishchuk:2003,BirkenhakeLange2004ComplexAbelianVarieties,Folland:1989}.

\subsection{Symplectic abelian groups}\label{6kiy7crh}

Let $K$ be a locally compact abelian group. We write its group law additively and denote its neutral element by $0 \in K$.

An \defn{alternating form} on $K$ is a continuous group bihomomorphism
\[
\omega : K \times K \too \U(1)
\]
such that $\omega(x, x) = 1$ for all $x \in K$.
A \defn{symplectic form} on $K$ is an alternating form $\omega$ that is \defn{non-degenerate}, i.e.\ if $\omega(x, y) = 1$ for all $y \in K$, then $x = 0$. 

A \defn{symplectic abelian group} is a pair $(K, \omega)$ consisting of a locally compact abelian group and a symplectic form on it.

For a sequence of positive integers $(d_1, \ldots, d_n)$, set
\[
D \coloneqq \diag(d_1, \ldots, d_n)
\]
and
\[
K_D \coloneqq (\Z^n/D\Z^n) \oplus (\Z^n/D\Z^n),
\]
where
\[
\Z^n/D\Z^n = \Z/d_1\Z \oplus \cdots \oplus \Z/d_n\Z.
\]
For a finite abelian group $K$, we denote by $\exp(K)$ its \defn{exponent}, i.e.\ the smallest integer $m \ge 1$ such that $mk = 0$ for all $k \in K$. In particular,
\[
\exp(K_D) = \operatorname{lcm}(d_1, \ldots, d_n).
\]

There is a canonical symplectic form on $K_D$ given by
\begin{equation}\label{8zd3oh5f}
\omega_D : K_D \times K_D \too \U(1), \quad \omega_D((a, b), (u,v)) = \exp(2\pi i \sum_{j = 1}^n \frac{b_j u_j - a_j v_j}{d_j}),
\end{equation}
and we refer to the pair $(K_D, \omega_D)$ as the \defn{standard symplectic abelian group} of type $D$.

\begin{lemma}\label{b47u02uc}
Let $(K, \omega)$ be a non-trivial finite symplectic abelian group.
Then there exist unique integers $d_1, \ldots, d_n \ge 2$ such that $d_1 \mid d_2 \mid \cdots \mid d_n$ and
\begin{equation}\label{sp66c0hd}
(K, \omega) \cong (K_D, \omega_D)
\end{equation}
as symplectic abelian groups, where $D = \diag(d_1, \ldots, d_n)$.
\end{lemma}

\begin{proof}
Let $m = \exp(K)$.
Choose $x \in K$ of order $m$. Since $\omega$ is non-degenerate, the character $\omega(x, \cdot) \in \widehat{K} \coloneqq \operatorname{Hom}(K, \U(1))$ has order $m$. Hence there exists $y \in K$ such that $\omega(x, y)$ has order $m$. Then $y$ has order $m$, and after replacing $y$ by a suitable multiple of itself, we may assume that
\[
\omega(x, y) = e^{-2\pi i/m}.
\]
Let $K_0 = \ip{x, y}$. If $kx + ly = 0$, then $1 = \omega(kx + ly, x) = \omega(y, x)^l$, so $m \mid l$. Similarly, $m \mid k$. Hence $K_0 \cong \Z/m\Z \oplus \Z/m\Z$. With respect to the basis $x, y$, the restriction of $\omega$ to $K_0$ is given by
\[
\omega((a,b),(u,v)) = \exp(2\pi i\frac{bu - av}{m}),
\]
so $K_0$ is a standard symplectic abelian group of type $(m)$.

Let
\[
K_0^\perp \coloneqq \{k \in K : \omega(k, h) = 1 \text{ for all } h \in K_0\}.
\]
Since the restriction of $\omega$ to $K_0$ is non-degenerate, we have $K_0 \cap K_0^\perp = 0$. Moreover, the map $K \to \widehat{K_0}$ given by $k \mtoo \omega(k, \cdot)|_{K_0}$ has kernel $K_0^\perp$ and is surjective. Thus $|K| = |K_0||K_0^\perp|$, and therefore
\[
K = K_0 \oplus K_0^\perp.
\]
The exponent of $K_0^\perp$ divides $m$, so iterating this construction gives the desired decomposition.

Finally, the integers $d_i$ are uniquely determined by the finite abelian group $K$, by the structure theorem for finite abelian groups.
\end{proof}

\subsection{Heisenberg groups}

We begin with a general construction. Consider a central extension
\[
\begin{tikzcd}
1 \arrow{r} & A \arrow{r}{f} & B \arrow{r}{g} & C \arrow{r} & 1,
\end{tikzcd}
\]
where $C$ is abelian. Since $A$ lies in the center of $B$, the commutator of two elements of $B$ only depends on their images in $C$. Thus there is a well-defined map
\begin{equation}\label{4skewf8i}
\omega : C \times C \too A
\end{equation}
uniquely characterized by
\[
f(\omega(g(b_1), g(b_2))) = b_1b_2b_1^{-1}b_2^{-1}
\]
for all $b_1, b_2 \in B$. This map is called the \defn{commutator map} of the central extension. It is an alternating form.

Let $(K, \omega)$ be a symplectic abelian group. A \defn{Heisenberg group} over $(K, \omega)$ is a central extension
\[
\begin{tikzcd}
1 \arrow{r} & \U(1) \arrow{r} & \H \arrow{r} & K \arrow{r} & 1
\end{tikzcd}
\]
whose commutator map is $\omega$. Such Heisenberg groups exist and are unique up to isomorphism of central extensions of $K$ by $\U(1)$.

The basic representation-theoretic fact about Heisenberg groups is the \defn{Stone--von Neumann theorem}. It says that $\H$ has a unique irreducible unitary representation
\[
\rho : \H \too \U(W)
\]
on which the central subgroup $\U(1) \subset \H$ acts by scalar multiplication. We call this representation the \defn{Stone--von Neumann representation} of $\H$. It is faithful, and the centralizer of $\rho(\H)$ is the group of scalars.

The Heisenberg group of a finite symplectic abelian group will be described explicitly in Section \ref{fospnkpi}.

\subsection{The real Heisenberg group and its representations}\label{p64bwipt}

Let $V$ be a real vector space of dimension $2n$ endowed with a symplectic form
\[
E : V \times V \too \R.
\]
The \defn{real Heisenberg group} of $(V,E)$ is the Heisenberg group of the symplectic abelian group $(V, \exp(-2\pi i E))$. Concretely, it is the group
\[
\HH(V) \coloneqq V \times \U(1)
\]
with multiplication
\[
(u, \alpha)(v, \beta) = \big(u + v,  \alpha\beta e^{-i \pi E(u, v)}\big)
\]
for all $(u, \alpha), (v, \beta) \in \HH(V)$.

We now recall two standard realizations of the Stone--von Neumann representation of $\HH(V)$, namely the Schr\"odinger and Fock representations, together with the Bargmann transform relating them.

\subsubsection{The Schr\"odinger representation}\label{4c5kvuhw}

Choose an identification $V = \R^{2n}$ such that $E$ is the standard symplectic form
\[
E : \R^{2n} \times \R^{2n} \too \R,
\]
given by
\[
E((x_1,y_1,\ldots, x_n, y_n), (\tilde{x}_1, \tilde{y}_1, \ldots, \tilde{x}_n, \tilde{y}_n))
=
\sum_{i = 1}^n (x_i \tilde{y}_i - y_i \tilde{x}_i).
\]
The \defn{Schr\"odinger representation} is the representation of $\HH(V)$ on the space $L^2(\R^n)$ of square-integrable functions $\R^n \to \C$ given by
\[
\HH(V) \too \U(L^2(\R^n)),\quad (u, \alpha) \mtoo \alpha T_u,
\]
where the \defn{translation operators} \cite[Eq.\ (1)]{GKP,royer2022encoding} are defined by
\begin{equation}\label{o88gpt48}
T_{(x_1, y_1, \ldots, x_n, y_n)} = \exp(-i\sqrt{2\pi}(y \cdot \hat{q} - x \cdot \hat{p})).
\end{equation}
Here $\hat{q} = (\hat{q}_1, \ldots, \hat{q}_n)$ and $\hat{p} = (\hat{p}_1, \ldots, \hat{p}_n)$ are the position and momentum operators acting on a function $\psi : \R^n \to \C$ by
\begin{equation}\label{xvpb3mus}
(\hat{q}_j\psi)(r) = r_j \psi(r),
\qquad
(\hat{p}_j\psi)(r) = \frac{1}{i}\frac{\partial\psi}{\partial r_j}(r),
\end{equation}
for $r \in \R^n$. Note that the $\sqrt{2\pi}$ factor in the definition of the translation operators is not standard in the broader physics and mathematics literature, but is convenient when working with GKP codes.

Equivalently, by the Baker--Campbell--Hausdorff formula,
\[
(T_{(x,y)}\psi)(r)
=
e^{i\pi x\cdot y}
e^{-i\sqrt{2\pi}\,y\cdot r}
\psi(r+\sqrt{2\pi}\,x),
\quad r \in \R^n.
\]
The translation operators satisfy
\begin{equation}\label{onkixs2f}
T_u T_v = e^{-2\pi i E(u, v)} T_v T_u = e^{-\pi i E(u, v)} T_{u + v}
\end{equation}
for all $u, v \in \R^{2n}$. The Schr\"odinger representation is irreducible and unitary, and is therefore a realization of the Stone--von Neumann representation of $\HH(V)$.

The Schr\"odinger representation has a unique continuous extension to the space $\S'(\R^n)$ of tempered distributions on $\R^n$.

\subsubsection{The Fock representation}\label{uwthy1do}

A second realization of the Stone--von Neumann representation of $\HH(V)$ is obtained after choosing a compatible complex structure on $V$. Choose a complex structure $I$ on $V$ such that
\[
E(Iu, Iv) = E(u, v)
\]
for all $u, v \in V$ and
\[
E(Iu, u) > 0
\]
for all $u \ne 0$. The corresponding hermitian inner product is
\[
H(z, w) = E(Iz, w) + i E(z, w).
\]
Regard $V$ as a complex vector space via $I$. Let $\mathcal{F}(V)$ be the space of holomorphic functions $f : V \to \C$ such that
\[
\int_V |f(z)|^2\exp(-\pi H(z, z))\vol_H < \infty,
\]
where $\vol_H$ is the standard volume form induced by $H$. Then $\mathcal{F}(V)$ is a Hilbert space with inner product
\[
\ip{f, g}_{\mathcal{F}} = \frac{1}{\pi^n}\int_V f(z)\overline{g(z)}\exp(-\pi H(z, z))\vol_H.
\]
The Heisenberg group $\HH(V)$ acts unitarily and irreducibly on $\mathcal{F}(V)$ by
\begin{equation}\label{jgbobzr8}
(\rho(u, \alpha)f)(z) = \alpha e^{-\pi(\frac{1}{2}H(u, u) + H(z, u))}f(z + u)
\end{equation}
for all $(u, \alpha) \in \HH(V)$ and $z \in V$. This is the \defn{Fock representation} of $\HH(V)$.

As for the Schr\"odinger representation, there is a natural extension of the Fock representation to a larger space. Let $\E(V)$ be the space of holomorphic functions $f : V \to \C$ of Gaussian-polynomial growth, meaning that there exist constants $k \ge 0$ and $C > 0$ such that
\begin{equation}\label{w19o40qe}
|f(v)| \le C e^{\pi H(v, v)/2}(1 + H(v, v))^{k/2}
\end{equation}
for all $v \in V$. Then the formula \eqref{jgbobzr8} defines a representation of $\HH(V)$ on $\E(V)$.

\subsubsection{The Bargmann transform}\label{73mes5iz}

By the Stone--von Neumann theorem, the Schr\"odinger and Fock representations are isomorphic as unitary representations of $\HH(V)$. An explicit isomorphism is given by the Bargmann transform \cite{Bargmann1961}. To write it down, choose a symplectic isomorphism $V \cong \R^{2n}$ such that the complex structure is the one induced by the identification
\[
\R^{2n} = \C^n, \qquad (x, y) = x - iy.
\]
Equivalently,
\[
I(x_1, y_1, \ldots, x_n, y_n) = (y_1, -x_1, \ldots, y_n, -x_n).
\]
In these coordinates,
\[
H(z, w) = z \cdot \bar{w}.
\]

The Bargmann transform is the unitary map
\[
\B : L^2(\R^n) \too \mathcal{F}(\C^n)
\]
defined by
\[
\B(\psi)(z) = \pi^{n/4} \int_{\R^n} e^{-\frac{1}{2}(\pi z \cdot z + q \cdot q) + \sqrt{2\pi} z \cdot q} \psi(q)dq
\]
for all $\psi \in L^2(\R^n)$ and $z \in \C^n$.
\footnote{We have introduced a rescaling by $\sqrt{\pi}$ so that $\B(\psi)(z) = \pi^{n/2} \B_{\mathrm{std}}(\psi)(\sqrt{\pi}z)$, where $\B_{\mathrm{std}}$ is the standard Bargmann transform as it appears originally in \cite{Bargmann1961,Bargmann1967}. This rescaling makes the transform an isomorphism of representations with our normalization of the translation operators.}
As shown in \cite{Bargmann1967}, the Bargmann transform extends uniquely to a linear homeomorphism
\begin{equation}\label{l5qky9yw}
\B : \S'(\R^n) \too \E(\C^n),
\end{equation}
where $\mathcal{S}'(\R^n)$ is the space of tempered distributions on $\R^n$ and $\E(\C^n)$ is the space of holomorphic functions of Gaussian-polynomial growth defined in Section \ref{uwthy1do}.

For an operator
\[
A : \mathcal{S}'(\R^n) \too \mathcal{S}'(\R^n),
\]
we denote by
\[
A^\B \coloneqq \B \circ A \circ \B^{-1}
\]
the corresponding operator on $\E(\C^n)$. In particular, consider the position and momentum operators $\hat{q}_j, \hat{p}_j : \S'(\R^n) \to \S'(\R^n)$ defined by \eqref{xvpb3mus}. Let $\hat{a}_j$ and $\hat{a}_j^\dagger$ be defined by
\[
\hat{q}_j = \frac{\hat{a}_j + \hat{a}_j^\dagger}{\sqrt{2}}
\quad\text{and}\quad
\hat{p}_j = \frac{\hat{a}_j - \hat{a}_j^\dagger}{i\sqrt{2}}.
\]
Then
\begin{equation}\label{5ctxjsq8}
(\hat{a}_j^\B f)(z) = \frac{1}{\sqrt{\pi}}\frac{\partial f}{\partial z_j}(z)
\quad\text{and}\quad
((\hat{a}_j^\dagger)^\B f)(z) = \sqrt{\pi} z_j f(z)
\end{equation}
for all $f \in \E(\C^n)$ and $z \in \C^n$ \cite{Bargmann1961,Bargmann1967}.

It follows that, if $u = x - iy \in \C^n$, then
\[
T_u^\B = \exp(\sqrt{\pi}(u \cdot \hat{a}^\B - \bar{u} \cdot (\hat{a}^\dagger)^\B)).
\]
By the Baker--Campbell--Hausdorff formula, we have
\begin{equation}\label{c03f536z}
(T^\B_uf)(z) = e^{-\pi(\frac{1}{2}H(u, u) + H(z, u))} f(z + u)
\end{equation}
for all $f \in \E(\C^n)$ and $z \in \C^n$.
Thus the Bargmann transform is an isomorphism of $\HH(V)$-representations.

\subsubsection{The extended metaplectic representation}\label{zz4171u8}

Let $\Sp(V)$ be the symplectic group of $(V, E)$ and let $\pi : \Mp(V) \to \Sp(V)$ be its double cover, where $\Mp(V)$ is the metaplectic group. Throughout this section and the rest of the paper, if $M \in \Mp(V)$ and $v \in V$, we write
\[
Mv \coloneqq \pi(M)v.
\]

Choose an identification $V \cong \R^{2n}$ as in Section \ref{4c5kvuhw} to define the Schr\"odinger representation on $L^2(\R^n)$. 
The \defn{metaplectic representation} is the unitary representation
\[
Q : \Mp(V) \too \U(L^2(\R^n)), \quad M \mtoo Q_M,
\]
characterized by the relation
\[
Q_M T_u Q_M^{-1} = T_{Mu}, \quad \text{for all } M \in \Mp(V) \text{ and } u \in V.
\]

The Schr\"odinger and metaplectic representations combine into a single unitary representation
\[
\rho : \Mp(V) \ltimes \HH(V) \too \U(L^2(\R^n)), \quad (M, u, \alpha) \mtoo \alpha T_u Q_M,
\]
where the semidirect product $\Mp(V) \ltimes \HH(V)$ is called the \defn{extended metaplectic group} and has group law
\[
(M, u, \alpha)(N, v, \beta)
=
(MN, u + Mv, \alpha \beta e^{-i\pi E(u, Mv)}).
\]
The representation $\rho$ extends uniquely to a continuous representation on $\mathcal{S}'(\R^n)$.

Recall that the restriction $\pi : \pi^{-1}(\U(V)) \to \U(V)$ splits, where $\U(V) = \Sp(V) \cap \mathrm{O}(V)$. Hence the extended metaplectic representation restricts to a representation
\[
\U(V) \ltimes \HH(V) \too \U(L^2(\R^n)).
\]
Under the Fock representation $\mathcal{F}(V)$, this restricted representation is given by
\begin{equation}\label{kilzod4v}
(\rho(M, u, \alpha) f)(z) = \alpha e^{-\pi(\frac{1}{2}H(u, u) + H(z, u))} f(M^{-1}(z + u));
\end{equation}
see \cite[Ch.\ 4, Section 2]{Folland:1989}.

\subsection{Symplectically integral lattices and complex abelian varieties}\label{23yhqu4v}

Let $V$, $E$, $I$, and $H$ be as in the preceding section.
For a lattice $\Lambda \subset V$, its \defn{symplectic dual} is the lattice
\[
\Lambda^\perp \coloneqq \{\mu \in V : E(\mu, \Lambda) \subset \Z\}.
\]
A lattice is called \defn{symplectically integral} if $\Lambda \subset \Lambda^\perp$. The following result is standard.

\begin{lemma}[Frobenius]\label{aden93im}
Let $\Lambda \subset V$ be a symplectically integral lattice of maximal rank. Then there is a basis $\lambda_1, \ldots, \lambda_n, \mu_1, \ldots, \mu_n$ for $\Lambda$ such that for all $i, j$ we have $E(\lambda_i, \lambda_j) = 0$, $E(\mu_i, \mu_j) = 0$, and $E(\lambda_i, \mu_j) = d_i \delta_{ij}$, where $d_i$ are positive integers satisfying $d_1 \mid d_2 \mid \cdots \mid d_n$.
Moreover, $(d_1, \ldots, d_n)$ is uniquely determined by $\Lambda$, independently of the choice of basis.
\end{lemma}

\begin{definition}\label{z2w9s5vx}
The sequence $(d_1, \ldots, d_n)$ in Lemma \ref{aden93im} is called the \defn{type} of $\Lambda$. A basis $\lambda_1, \ldots, \lambda_n, \mu_1, \ldots, \mu_n$ as in Lemma \ref{aden93im} is called a \defn{symplectic basis}.
\end{definition}

In this case,
\[
K \coloneqq \Lambda^\perp/\Lambda
\]
is a finite symplectic abelian group with respect to
\[
\omega : K \times K \too \U(1), \quad \omega([\mu_1], [\mu_2]) = e^{-2\pi i E(\mu_1, \mu_2)}.
\]

\begin{lemma}\label{jiizju0l}
The exponent of $K$ is $d_n$.
\end{lemma}

\begin{proof}
The symplectic dual lattice is generated by $\frac{\lambda_1}{d_1}, \ldots, \frac{\lambda_n}{d_n}, \frac{\mu_1}{d_1}, \ldots, \frac{\mu_n}{d_n}$. Hence the exponent of $K$ is the smallest integer $m$ such that $m/d_i \in \Z$ for all $i$, i.e.\ $m = \mathrm{lcm}(d_1, \ldots, d_n) = d_n$.
\end{proof}

A \defn{semicharacter} for $(\Lambda, E)$ is a map $\nu : \Lambda \to \U(1)$ such that
\[
\nu(\lambda + \mu) = \nu(\lambda)\nu(\mu)e^{i\pi E(\lambda, \mu)}
\]
for all $\lambda, \mu \in \Lambda$. In this case, the map
\[
\Lambda \too \HH(V), \quad \lambda \mtoo (\lambda, \nu(\lambda)^{-1})
\]
is a group homomorphism from the abelian group $\Lambda$ to the real Heisenberg group $\HH(V)$.
We denote its image by
\begin{equation}\label{bhl8ig4v}
\Lambda_\nu \coloneqq \{(\lambda, \nu(\lambda)^{-1}) \in \HH(V) : \lambda \in \Lambda\}.
\end{equation}
Let $C_{\HH(V)}(\Lambda_\nu)$ be the centralizer of $\Lambda_\nu$ in $\HH(V)$.
Concretely, we have
\[
C_{\HH(V)}(\Lambda_\nu) = \Lambda^\perp \times \U(1) \subset \HH(V).
\]
The Heisenberg group of $(K, \omega)$ can then be identified with the quotient group
\[
\Heis_{\Lambda, \nu} \coloneqq C_{\HH(V)}(\Lambda_\nu)/\Lambda_\nu.
\]

The Stone--von Neumann representation of $\Heis_{\Lambda, \nu}$ can be obtained from that of $\HH(V)$ as follows. Let $\rho : \HH(V) \too \GL(\E(V))$ be the representation of $\HH(V)$ on the space $\E(V)$ defined in Section \ref{uwthy1do}. The space of \defn{canonical theta functions} is the stabilizer of $\Lambda_\nu$, namely
\begin{align*}
\Theta_{\Lambda, \nu} &\coloneqq \{f \in \E(V) : \rho(\lambda, \nu(\lambda)^{-1})f = f  \text{ for all } \lambda \in \Lambda\} \\
&= \{ f \in \E(V) : f(z + \lambda) = \nu(\lambda) e^{\pi(\frac{1}{2}H(\lambda, \lambda) + H(z, \lambda))} f(z) \text{ for all } z \in V \text{ and } \lambda \in \Lambda\}.
\end{align*}
It is naturally a representation of $\Heis_{\Lambda, \nu}$. Moreover, this representation is unitary with respect to
\begin{equation}\label{cu7nnb56}
\ip{f, g}_\Theta \coloneqq \frac{1}{\pi^n}\int_{V/\Lambda} f(v)\overline{g(v)}\exp(-\pi H(v, v)) \vol_H,
\end{equation}
where $\vol_H$ is the volume form on $V/\Lambda$ induced by $H$.
The representation $\Theta_{\Lambda, \nu}$ is irreducible, and hence is the Stone--von Neumann representation of $\Heis_{\Lambda, \nu}$.

Recall that the lattice $\Lambda$ defines a complex abelian variety $X \coloneqq V/\Lambda$, and the semicharacter $\nu$ defines an ample line bundle $L \coloneqq L(H, \nu)$ on $X$ via the Appell--Humbert theorem. Explicitly,
\begin{equation}\label{zcsp4rse}
L(H, \nu) \coloneqq (V \times \C) / \Lambda,
\end{equation}
where $\Lambda$ acts by
\begin{equation}\label{y5vauty8}
\lambda \cdot (v, z) = (v + \lambda,\, \nu(\lambda)e^{\pi(\frac{1}{2}H(\lambda, \lambda) + H(v, \lambda))}z).
\end{equation}
It follows that every section of $L$ is of the form $[v] \mto [v, f(v)]$ for a unique theta function $f \in \Theta_{\Lambda, \nu}$, giving a canonical identification
\[
\Theta_{\Lambda, \nu} \cong H^0(X, L).
\]

We now relate the quotient Heisenberg group above to the theta group of the line bundle $L$. For $x \in X$, we denote the translation by $x$ by
\[
t_x : X \too X, \quad y \mtoo x + y.
\]
The \defn{theta group} of $L$, denoted $\G(L)$, is the group of line bundle isomorphisms $\varphi : L \to L$ covering a translation $t_x : X \to X$. The theta group has a natural action on $H^0(X, L)$ by $\varphi \cdot s = \varphi \circ s \circ t_{-x}$, for every $\varphi \in \G(L)$ covering $t_x$ and $s \in H^0(X, L)$.

The line bundle $L$ has a canonical hermitian fibre metric, given on each fibre $L_x = \{[(v, z)] : z \in \C\}$ by
\[
\ip{[(v, z)], [(v, w)]} = z\bar{w} e^{-\pi H(v, v)}.
\]
The \defn{compact theta group} is the subgroup of elements of $\G(L)$ preserving this metric, denoted $\G_c(L)$.

If $\varphi \in \G(L)$ covers $t_x$, then $x \in K = \Lambda^\perp/\Lambda$. Moreover, the sequence
\[
1 \too \U(1) \too \G_c(L) \too K \too 1
\]
is a central extension with commutator map $\omega(x, y) = \exp(-2\pi i E(x, y))$. It follows that $\G_c(L) \cong \Heis_{\Lambda, \nu}$.

An explicit identification can be obtained as follows. For $(\mu, \alpha) \in \Lambda^\perp \times \U(1) \subset \HH(V)$, let
\[
\varphi_{\mu, \alpha} : V \times \C \too V \times \C, \quad \varphi_{\mu, \alpha}(v, z) = (v - \mu, \alpha e^{\pi(\frac{1}{2}H(\mu, \mu) - H(v, \mu))}z).
\]
We have $\varphi_{\mu, \alpha}(\lambda \cdot (v, z)) = \lambda \cdot \varphi_{\mu, \alpha}(v, z)$ for all $\lambda \in \Lambda$, so that $\varphi_{\mu, \alpha}$ descends to an automorphism
\[
\varphi_{\mu, \alpha} : L \too L.
\]
Moreover,
\[
\varphi_{\mu, \alpha} \circ \varphi_{\tilde{\mu}, \tilde{\alpha}} = \varphi_{(\mu, \alpha)(\tilde{\mu}, \tilde{\alpha})}.
\]
It is also easy to see that $\varphi_{\lambda, \nu(\lambda)^{-1}} = \mathrm{Id}_L$ for all $\lambda \in \Lambda$. Since $\Heis_{\Lambda, \nu} = (\Lambda^\perp \times \U(1))/\Lambda_\nu$, we get a group isomorphism
\[
\Heis_{\Lambda, \nu} \too \G_c(L), \quad [\mu, \alpha] \mtoo \varphi_{\mu, \alpha}
\]
restricting to the identity on $\U(1)$.

\section{Pauli and Clifford groups}\label{tj099yyk}

In this section we recall the Pauli and Clifford groups and develop a basis-independent formulation of them. We first review the standard construction for qudits, mainly to fix the conventions and phases used later. We then show that the same construction can be expressed intrinsically in terms of a finite symplectic abelian group, its Heisenberg group, and the corresponding Stone--von Neumann representation. This point of view removes the dependence on a choice of coordinates or tensor product decomposition, and is the form that will be used for GKP codes and for the theta groups of polarized abelian varieties.

\subsection{The standard Pauli and Clifford groups on qudits}\label{fospnkpi}

For a positive integer $d$, the \defn{qudit} is the Hilbert space $\C^d$ together with its standard orthonormal basis, denoted $\ket{0}, \ket{1}, \ldots, \ket{d-1}$. Define unitary operators
\[
X : \C^d \too \C^d, \qquad X\ket{j} = \ket{j + 1} \text{ for all } j,
\]
where the index $j$ is considered modulo $d$, and
\[
Z : \C^d \too \C^d, \qquad Z\ket{j} = e^{2\pi ij/d} \ket{j} \text{ for all } j.
\]
When $d = 2$, these recover the standard Pauli operators on $\C^2$. For a general $d$, we have
\begin{equation}\label{7r7hxjs8}
    X^d = Z^d = 1, \qquad ZX = e^{2\pi i/d}XZ.
\end{equation}
The \defn{Pauli group} on $\C^d$ \cite{HostensDehaeneDeMoor2005} is the subgroup of $\U(\C^d)$ generated by $X$, $Z$, and $e^{\pi i/d}$, denoted
\[
\Pauli_d \coloneqq \ip{X, Z, e^{\pi i/d}}.
\]
It is a finite group of order $2d^3$.

\begin{remark}
It is also natural to consider the smaller group generated only by $X$ and $Z$, which contains $e^{2\pi i/d}$ but not necessarily $e^{\pi i/d}$. However, it is standard to include $e^{\pi i/d}$ as a generator in order to recover the $Y$-gate when $d = 2$, namely $Y = iXZ = e^{\pi i/2}XZ$. This convention also leads to a more intrinsic definition; see Section \ref{hcdib98d}.
\end{remark}

More generally, let $d_1, \ldots, d_n$ be positive integers and set 
\[
D \coloneqq \diag(d_1, \ldots, d_n).
\]
Consider the system of qudits
\[
\C^{\otimes D} \coloneqq \C^{d_1} \otimes \cdots \otimes \C^{d_n}.
\]
Each factor has its own Pauli operators
\[
X_i = I \otimes \cdots \otimes I \otimes X \otimes I \otimes \cdots \otimes I, \qquad Z_i = I \otimes \cdots \otimes I \otimes Z \otimes I \otimes \cdots \otimes I,
\]
for $i = 1, \ldots, n$, where $X$ and $Z$ appear in the $i$th position. Let
\[
m \coloneqq \operatorname{lcm}(d_1, \ldots, d_n)
\]
be the least common multiple of $d_1, \ldots, d_n$.
Then $m$ is the smallest integer such that the set
\[
\{e^{2\pi i c/m}X_1^{a_1}\cdots X_n^{a_n}Z_1^{b_1}\cdots Z_n^{b_n} : a_i, b_i, c \in \Z\}
\]
is closed under multiplication.
As in the single-qudit case, we add a square root of $e^{2\pi i/m}$ to the set of generators.

\begin{definition}
The \defn{standard Pauli group} associated with $D$ is the group
\[
\Pauli_{D} \coloneqq \ip{X_1, \ldots, X_n, Z_1, \ldots, Z_n, e^{i\pi/m}}
\]
generated by $X_1, \ldots, X_n, Z_1, \ldots, Z_n$, and $e^{i\pi/m}$.
\end{definition}

By \eqref{7r7hxjs8}, every element of $\Pauli_{D}$ can be written uniquely as
\begin{equation}\label{4k8itljs}
e^{c\pi i/m} X_1^{a_1}\cdots X_n^{a_n}Z_1^{b_1}\cdots Z_n^{b_n},
\end{equation}
for $c \in \Z/2m\Z$ and $a_i, b_i \in \Z/d_i\Z$.
In particular, the Pauli group is finite of order
\[
|\Pauli_{D}| = 2m(d_1\cdots d_n)^2.
\]

Consider the standard symplectic abelian group $(K_D, \omega_D)$ as in Section \ref{6kiy7crh}. For all $x = (a, b) \in K_D$, set
\[
P_x \coloneqq X_1^{a_1}\cdots X_n^{a_n}Z_1^{b_1}\cdots Z_n^{b_n}.
\]
We then have the commutator relation
\begin{equation}\label{af95ua00}
P_x P_y = \omega_D(x, y) P_y P_x,    
\end{equation}
for all $x, y \in K_D$.

\begin{definition}
The \defn{standard Heisenberg group} associated with $D$ is the group
\[
\Heis_{D} \coloneqq \{ \alpha X_1^{a_1}\cdots X_n^{a_n}Z_1^{b_1}\cdots Z_n^{b_n} : \alpha \in \U(1) \text{ and } (a, b) \in K_D\} = \U(1)\Pauli_{D}.
\]
\end{definition}

It follows from \eqref{af95ua00} that $\Heis_{D}$ is the Heisenberg group of $(K_D, \omega_D)$. The central extension
\[
\begin{tikzcd}
    1 \arrow{r} & \U(1) \arrow{r} & \Heis_{D} \arrow{r} & K_D \arrow{r} & 1
\end{tikzcd}
\]
is given by the natural inclusion $\U(1) \subset \U(\C^{\otimes D})$ and the map sending $\alpha P_x \in \Heis_{D}$ to $x \in K_D$.
The Stone--von Neumann representation of $\Heis_{D}$ is the inclusion $\Heis_{D} \subset \U(\C^{\otimes D})$.

The Pauli group also has the following intrinsic characterization.

\begin{proposition}\label{b85yc39w}
We have
\begin{equation}\label{vp3muef8}
\Pauli_{D} = \{h \in \Heis_{D} : h^{2m} = 1\}.
\end{equation}
\end{proposition}

\begin{proof}
Let $h = \alpha X_1^{a_1}\cdots X_n^{a_n}Z_1^{b_1}\cdots Z_n^{b_n} \in \Heis_{D}$. By \eqref{af95ua00}, applied recursively with $P_x = P_{(a, 0)}P_{(0, b)}$, we get
\[
h^{2m} = \alpha^{2m} \omega_D((0,b),(a,0))^{m(2m-1)} = \alpha^{2m},
\]
since $\omega_D((0,b),(a,0))^m = 1$. It follows that $\alpha \in \ip{e^{\pi i/m}}$ if and only if $h^{2m} = 1$, proving \eqref{vp3muef8}.
\end{proof}

\begin{definition}
The \defn{standard Clifford group} associated with $D$ is the normalizer of $\Pauli_{D}$ in $\U(\C^{\otimes D})$, denoted 
\[
\Cliff_{D} \coloneqq N_{\U(\C^{\otimes D})}(\Pauli_{D}).
\]
\end{definition}

An element $U \in \Cliff_{D}$ induces an automorphism
\[
\Pauli_{D} \too \Pauli_{D}, \qquad P \mtoo UPU^\dagger.
\]
This automorphism determines $U$ up to a phase:

\begin{proposition}
If $U, V \in \Cliff_{D}$ satisfy $UPU^\dagger = VPV^\dagger$ for all $P \in \Pauli_{D}$, then $U = e^{i\theta}V$ for some $\theta \in \R$.
\end{proposition}

The proof will be given later in a more abstract setting; see Proposition \ref{94mzk44o}.
A Clifford gate $U \in \Cliff_D$ also determines a map
\[
\bar{U} : K_D \too K_D
\]
such that
\[
UP_xU^\dagger = (\mathrm{phase})P_{\bar{U}x}
\]
for all $x \in K_D$. One checks that
\[
\omega_D(\bar{U}x, \bar{U}y) = \omega_D(x, y)
\]
for all $x, y \in K_D$, i.e.\ $\bar{U}$ is an element of the symplectic group
\[
\Sp(K_D) \coloneqq \{\varphi \in \Aut(K_D) : \omega_D(\varphi x, \varphi y) = \omega_D(x, y) \text{ for all } x, y \in K_D\}.
\]
It can be shown that this map determines $U$ up to a phase and a Pauli gate and, moreover, that every element of $\Sp(K_D)$ is obtained in this way; see Proposition \ref{xor7ubv5}. In other words:

\begin{theorem}
We have
\[
\Cliff_{D}/\Heis_{D} \cong \Sp(K_D).
\]
\end{theorem}

\subsection{The abstract Pauli and Clifford groups}\label{hcdib98d}

The concepts and results on Pauli and Clifford groups just recalled admit a more general formulation that does not rely on choosing a decomposition of the qudit system into factors $\C^{d_i}$, nor on decomposing $K$ as $(\Z^n/D\Z^n) \oplus (\Z^n/D\Z^n)$. This is analogous to abstract linear algebra: although every finite-dimensional complex vector space is isomorphic to $\C^n$ for some $n$, it is nonetheless preferable to define vector spaces and linear maps in a basis-independent manner. In the same spirit, we introduce Pauli and Clifford groups associated with a finite symplectic abelian group $(K, \omega)$ and a Heisenberg group over it. The concrete description above is then recovered after making certain noncanonical choices. A closely related approach appears in \cite{van2013efficient}, which can be interpreted as the special case of the present formalism where $K = G \oplus \widehat{G}$, for $G$ a finite abelian group, equipped with its canonical symplectic form.

Let $(K, \omega)$ be a finite symplectic abelian group with Heisenberg group $\H$, and let $m$ be the exponent of $K$. 

\begin{definition}
The \defn{(abstract) Pauli group} associated to $\H$ is
\[
\P \coloneqq \{h \in \H : h^{2m} = 1\}.
\]
\end{definition}

\begin{proposition}
The group $\P$ is a finite subgroup of $\H$ of order $2m|K|$.
\end{proposition}

\begin{proof}
By Lemma \ref{b47u02uc}, we have an isomorphism $(K, \omega) \cong (K_D, \omega_D)$ for some $D = \diag(d_1, \ldots, d_n)$, and hence an isomorphism $\H \cong \Heis_D$ of central extensions. It follows from Proposition \ref{b85yc39w} that $\P$ corresponds to the standard Pauli group $\Pauli_D$ under this isomorphism. In particular, $\P$ is a finite subgroup of order $2m(d_1\cdots d_n)^2 = 2m|K|$.
\end{proof}

Let
\[
\rho : \H \too \U(W)
\]
be the Stone--von Neumann representation of $\H$. Since $\rho$ is faithful, we identify $\H$ with its image in $\U(W)$.

\begin{definition}
The \defn{(abstract) Clifford group} associated to $(K, \omega)$ is the normalizer
\[
\mathcal{C} \coloneqq N_{\U(W)}(\H)
\]
of $\H$ in $\U(W)$.
\end{definition}

Since $\U(1) \subset \H$ acts by scalars, $\mathcal{C}$ can equivalently be described as the normalizer of the Pauli group $\P \subset \H$.

Let $U \in \mathcal{C}$. Then conjugation by $U$ induces an automorphism $\varphi_U : \H \to \H$ restricting to the identity on $\U(1)$. Let $\Aut_{\U(1)}(\H)$ be the group of automorphisms of $\H$ restricting to the identity on $\U(1)$. We then have a sequence
\begin{equation}\label{m9h2wajm}
\begin{tikzcd}[row sep=0pt]
1 \arrow{r} & \U(1) \arrow{r} & \mathcal{C} \arrow{r} & \Aut_{\U(1)}(\H) \arrow{r} & 1. \\
& & U \arrow[mapsto]{r} & \varphi_U
\end{tikzcd}
\end{equation}

\begin{proposition}\label{94mzk44o}
The sequence \eqref{m9h2wajm} is exact.
\end{proposition}

\begin{proof}
To show exactness at $\mathcal{C}$, let $U \in \mathcal{C}$ be such that $\varphi_U = \mathrm{Id}_\H$. Then $U$ commutes with every element of $\H$, so by irreducibility and Schur's lemma, $U \in \U(1)$.

For surjectivity, let $\varphi \in \Aut_{\U(1)}(\H)$. Then $\rho \circ \varphi$ is another irreducible unitary representation of $\H$ on which $\U(1)$ acts by scalar multiplication. By the Stone--von Neumann theorem, there exists $U \in \U(W)$ such that $U \rho(\varphi(h)) = \rho(h)U$ for all $h \in \H$. It follows that $U \in \mathcal{C}$ and $\varphi_U = \varphi$.
\end{proof}

Since $\H/\U(1) \cong K$, every automorphism $\varphi_U$ descends to an automorphism of $K$ preserving the symplectic form $\omega$. Let $\Sp(K)$ be the group of automorphisms of $K$ preserving $\omega$. We obtain a sequence
\begin{equation}\label{rnumb4qb}
1 \too \H \too \mathcal{C} \too \Sp(K) \too 1.
\end{equation}

\begin{proposition}\label{xor7ubv5}
The sequence \eqref{rnumb4qb} is exact.
\end{proposition}

\begin{proof}
It is known (see e.g.\ \cite[Lemma 6.6]{BirkenhakeLange2004ComplexAbelianVarieties}) that the bottom row of the commutative diagram
\[
\begin{tikzcd}
& 1 \arrow{d} & 1 \arrow{d} & 1 \arrow{d} \\
1 \arrow{r} & \U(1) \arrow{d} \arrow[equals]{r} & \U(1) \arrow{d} \arrow{r} & 1 \arrow{d} \arrow{r} & 1 \\
1 \arrow{r} & \H \arrow{r} \arrow{d} & \mathcal{C} \arrow{r} \arrow{d} & \Sp(K) \arrow{r}\arrow[equals]{d} & 1 \\
1 \arrow{r} & K \arrow{r} \arrow{d} & \Aut_{\U(1)}(\H) \arrow{d} \arrow{r} & \Sp(K) \arrow{r} \arrow{d} & 1 \\
& 1 & 1 & 1
\end{tikzcd}
\]
is exact. Since all columns and the top row are exact, the middle row is also exact.
\end{proof}

As in the standard case, we therefore have
\[
\mathcal{C}/\H \cong \Sp(K).
\]

We now compare the abstract construction with the standard one.

\begin{proposition}\label{3ujxcjsa}
Let $(K, \omega)$ be a finite symplectic abelian group with Heisenberg group $\H$, and let $W$ be the Stone--von Neumann representation of $\H$.
There exist integers $d_1, \ldots, d_n \ge 2$ and an isomorphism
\begin{equation}\label{fejdalwy}
(K, \omega) \cong (K_D, \omega_D)
\end{equation}
of symplectic abelian groups, where $D = \diag(d_1, \ldots, d_n)$. Any such isomorphism lifts to an isomorphism
\begin{equation}\label{n0nrk8uq}
\Heis_D \cong \H
\end{equation}
restricting to the identity on $\U(1)$.
Given \eqref{n0nrk8uq}, there is a unitary isomorphism
\begin{equation}\label{mm1wz5yh}
F : \C^{\otimes D} \overset{\cong}{\too} W,
\end{equation}
unique up to a global $\U(1)$-scalar, such that the map
\begin{equation}\label{qx3xvyk0}
\U(\C^{\otimes D}) \too \U(W), \quad U \mtoo F \circ U \circ F^{-1}
\end{equation}
restricts to \eqref{n0nrk8uq}. In that case, \eqref{qx3xvyk0} also restricts to group isomorphisms
\[
\Pauli_D \cong \P \quad\text{and}\quad \Cliff_D \cong \mathcal{C},
\]
where $\P$ and $\mathcal{C}$ are the Pauli and Clifford groups associated to $(\H,W)$.
\end{proposition}

\begin{proof}
The existence of the integers $d_1, \ldots, d_n$ is Lemma \ref{b47u02uc}. Since $\H$ and $\Heis_D$ are Heisenberg groups over isomorphic finite symplectic abelian groups, the isomorphism \eqref{fejdalwy} lifts to an isomorphism \eqref{n0nrk8uq} of central extensions. Finally, since $W$ and $\C^{\otimes D}$ are irreducible unitary representations on which $\U(1)$ acts by scalar multiplication, the Stone--von Neumann theorem gives a unitary isomorphism $F : \C^{\otimes D} \to W$, unique up to a global scalar, intertwining the two representations. The induced map \eqref{qx3xvyk0} therefore identifies the corresponding Pauli and Clifford groups.
\end{proof}

\begin{remark}
The isomorphism \eqref{mm1wz5yh} can be made more explicit as follows.
Let $\Phi : K_D \to K$ be the chosen isomorphism of symplectic abelian groups, and let $e_i, f_i \in K_D$ be the standard generators corresponding to the two factors of $(\Z^n/D\Z^n) \oplus (\Z^n/D\Z^n)$. Choose lifts $\mathcal{X}_i, \mathcal{Z}_i \in \H$ of $\Phi(e_i)$ and $\Phi(f_i)$, respectively. After multiplying each lift by a suitable scalar, we may assume that $\mathcal{X}_i^{d_i} = \mathcal{Z}_i^{d_i} = 1$ for all $i$. 
These choices determine an isomorphism $\Heis_D \cong \H$ by sending
\[
\alpha X_1^{a_1}\cdots X_n^{a_n}Z_1^{b_1}\cdots Z_n^{b_n}
\mtoo
\alpha \mathcal{X}_1^{a_1}\cdots \mathcal{X}_n^{a_n}\mathcal{Z}_1^{b_1}\cdots \mathcal{Z}_n^{b_n}.
\]
Now, there exists a unit norm vector $w_0 \in W_0$, unique up to scaling, such that $\mathcal{Z}_iw_0 = w_0$ for all $i$. Then the map
\[
\C^{\otimes D} \too W, \quad \ket{b} \mtoo \mathcal{X}_1^{b_1}\cdots \mathcal{X}_n^{b_n}w_0,
\]
where $\ket{b} = \ket{b_1} \otimes \cdots \otimes \ket{b_n}$ and $b_i \in \Z/d_i\Z$, is the desired isomorphism. We do not need this result in the present paper, so we leave the details to the reader.
\end{remark}

\section{Definition of GKP codes}\label{ltg2p7xl}

The goal of this section is to give a precise mathematical definition of GKP codes in the language developed above. We begin by recalling the general structure of a quantum error-correcting code: a logical space, a physical space, an encoding, logical gates, and an error-correction procedure. We then specialize this discussion to the Schr\"odinger representation of the real Heisenberg group. In this setting, a GKP code is determined by a symplectically integral lattice $\Lambda$ together with a semicharacter $\nu$, and its code space is the simultaneous eigenspace of the corresponding translation operators in $\S'(\R^n)$. Finally, using the Bargmann transform, we identify this code space with the space of canonical theta functions and reformulate the construction geometrically in terms of polarized complex abelian varieties.

\subsection{Quantum error correction}

In quantum computation, information is represented by states in a finite-dimensional Hilbert space $\mathscr{H}_{\mathrm{logical}}$, while algorithms are modeled by sequences of unitary operators on $\mathscr{H}_{\mathrm{logical}}$, called logical gates, together with measurements. In practice, one often works with systems equipped with a tensor product decomposition
\[
\mathscr{H}_{\mathrm{logical}} = \C^{d_1} \otimes \cdots \otimes \C^{d_n},
\]
corresponding to a collection of qudits, most commonly qubits, for which $d_i = 2$. Standard gate sets include, for instance, Pauli and Clifford gates, together with more general unitary operations.

A central challenge is to realize $\mathscr{H}_{\mathrm{logical}}$ and its operations inside a physical system in a manner that is robust against noise. The purpose of quantum error correction is to do this by encoding the logical Hilbert space into a physical Hilbert space $\mathscr{H}_{\mathrm{physical}}$. Thus one considers a linear embedding
\begin{equation}\label{s9yrqkp6}
\operatorname{Enc} : \mathscr{H}_{\mathrm{logical}} \too \mathscr{H}_{\mathrm{physical}},
\end{equation}
whose image
\[
\mathscr{H}_{\mathrm{code}} \coloneqq \operatorname{Enc}(\mathscr{H}_{\mathrm{logical}})
\]
is called the \defn{code space}. Ideally, such an encoding should satisfy the following properties:
\begin{enumerate}[label={(P\arabic*)}]
    \item \label{0n07smer} (\emph{Isometric encoding}) The embedding \eqref{s9yrqkp6} is an isometry.
    
    \item \label{adfuq8hq} (\emph{Logical gate implementation}) A sufficiently large and relevant class of logical gates on $\mathscr{H}_{\mathrm{logical}}$ can be implemented by physically realizable unitary operators on $\mathscr{H}_{\mathrm{physical}}$ that preserve $\mathscr{H}_{\mathrm{code}}$.
    
    \item \label{bd6ih4k1} (\emph{Error correction}) If a sufficiently small noise process acts on a state $\psi \in \mathscr{H}_{\mathrm{code}}$, producing a perturbed state, then one can perform a measurement on $\mathscr{H}_{\mathrm{physical}}$, followed by a suitable unitary correction, so as to recover the original state $\psi$.
\end{enumerate}

For GKP codes, this general picture has to be interpreted with some care. The ideal physical space is naturally taken to be $\mathscr{H}_{\mathrm{physical}} = \mathcal{S}'(\R^n)$, since the ideal encoded states are typically not square-integrable. However, $\mathcal{S}'(\R^n)$ is not a Hilbert space, so the usual isometry requirement does not directly apply. We will therefore replace it by an asymptotic isometry condition, while keeping the same overall structure of logical space, physical space, encoding, gates, and error correction. These three properties will be treated in Sections \ref{b9j929je}, \ref{5nk36oj5}, and \ref{k3dqdqez}, respectively.

\subsection{GKP codes}

Let $(V,E)$ be a real symplectic vector space of dimension $2n$, together with an identification $V \cong \R^{2n}$ as in Section \ref{4c5kvuhw}. We use this identification to realize the real Heisenberg group $\HH(V)$ on $L^2(\R^n)$ by the Schr\"odinger representation, with translation operators $T_v$ as in \eqref{o88gpt48}.

The construction of GKP codes \cite{GKP,royer2022encoding,conrad2022gottesman} is modeled on the stabilizer formalism from quantum error correction \cite{GottesmanThesis,NielsenChuang}. For our purposes, this means that one seeks a code space as a simultaneous eigenspace of a family of commuting unitary operators. In the present setting, the relevant operators are the translations $T_v$. By the commutation relation \eqref{onkixs2f}, the translations $T_\lambda$ and $T_\mu$ commute precisely when $E(\lambda,\mu) \in \Z$. Thus the natural indexing set for such a commuting family is a symplectically integral lattice $\Lambda \subset V$.

One must also prescribe the corresponding eigenvalues. Suppose that a state $\psi$ satisfies $T_\lambda \psi = \nu(\lambda)\psi$ for all $\lambda \in \Lambda$. Since
\[
T_\lambda T_\mu = e^{-\pi i E(\lambda,\mu)}T_{\lambda+\mu},
\]
the function $\nu : \Lambda \to \U(1)$ must satisfy
\[
\nu(\lambda + \mu) = \nu(\lambda)\nu(\mu)e^{i\pi E(\lambda,\mu)}.
\]
Thus the eigenvalue data are exactly a semicharacter for $(\Lambda,E)$.

There is, however, one important point. In general, the simultaneous eigenspace cut out by the equations $T_\lambda \psi = \nu(\lambda)\psi$ for all $\lambda \in \Lambda$ does not lie in $L^2(\R^n)$. To obtain the ideal code states, one must therefore solve these equations in the larger space $\S'(\R^n)$ of tempered distributions. These states should be understood as idealized limits of physical, normalizable states. This leads to the following definition.

\begin{definition}
A \defn{GKP code} is a pair $(\Lambda, \nu)$ where $\Lambda \subset V$ is a symplectically integral lattice of rank $2n$ and $\nu : \Lambda \to \U(1)$ is a semicharacter. The \defn{type} of the GKP code $(\Lambda, \nu)$ is the type $(d_1, \ldots, d_n)$ of the lattice $\Lambda$ in Definition \ref{z2w9s5vx}.

The \defn{stabilizer group} of $(\Lambda,\nu)$ is the abelian group
\[
S_{\Lambda, \nu} \coloneqq \{\nu(\lambda)^{-1}T_\lambda : \lambda \in \Lambda\},
\]
i.e.\ the image of the abelian subgroup $\Lambda_\nu \subset \HH(V)$ under the Schr\"odinger representation; see Section \ref{23yhqu4v}.

The \defn{code space} of $(\Lambda, \nu)$ is the simultaneous $+1$ eigenspace of $S_{\Lambda, \nu}$ in $\S'(\R^n)$, denoted
\begin{equation}\label{owlt98ay}
\code_{\Lambda, \nu} \coloneqq \{\psi \in \S'(\R^n): T_\lambda \psi = \nu(\lambda) \psi \text{ for all } \lambda \in \Lambda\}.
\end{equation}

The \defn{Heisenberg group} of $(\Lambda, \nu)$ is the quotient group
\[
\Heis_{\Lambda, \nu} \coloneqq C_{\HH(V)}(\Lambda_\nu)/\Lambda_\nu,
\]
viewed as the Heisenberg group of $K = \Lambda^\perp/\Lambda$ with respect to $\exp(-2\pi i E)$; see Section \ref{23yhqu4v}. The \defn{Pauli} and \defn{Clifford groups} of $(\Lambda, \nu)$, denoted $\Pauli_{\Lambda, \nu}$ and $\Cliff_{\Lambda, \nu}$, respectively, are the corresponding Pauli and Clifford groups in the sense of Section \ref{hcdib98d}.
\end{definition}

\begin{remark}
In the physics literature, the semicharacter $\nu$ is called the \emph{gauge} \cite[Section II(C)]{royer2022encoding}. The choice of gauge is often implicitly encoded in a choice of basis $s_1, \ldots, s_{2n}$ for $\Lambda$, by taking the unique semicharacter satisfying $\nu(s_i)=1$ for all $i$, namely
\[
\nu(m_1s_1 + \cdots + m_{2n}s_{2n})
   = (-1)^{\sum_{i < j} m_i m_jE(s_i, s_j)},
   \qquad \text{for all } m_1, \ldots, m_{2n} \in \Z;
\]
see \cite[Eq.\ (30)]{royer2022encoding}.
\end{remark}

Let $m = d_n$, or equivalently (Lemma \ref{jiizju0l}), $m = \exp(\Lambda^\perp/\Lambda)$. In particular, $m\mu \in \Lambda$ for all $\mu \in \Lambda^\perp$. The intrinsic definition of the Pauli group therefore gives the following explicit description.

\begin{lemma}\label{g8ok30mq}
The Pauli group of $(\Lambda,\nu)$ is given by
\begin{equation}\label{b6fwcvci}
\Pauli_{\Lambda, \nu}
=
\{[\mu, \alpha] \in \Heis_{\Lambda, \nu} : \mu \in \Lambda^\perp \text{ and } \alpha^{2m} = \nu(m\mu)^{-2}\}.
\end{equation}
\end{lemma}

\begin{proof}
By definition, $[\mu,\alpha] \in \Pauli_{\Lambda,\nu}$ if and only if $[\mu,\alpha]^{2m}=1$ in $\Heis_{\Lambda,\nu}$. Since $E(\mu,\mu)=0$, this is equivalent to
\[
(2m\mu,\alpha^{2m}) \in \Lambda_\nu.
\]
By the definition of $\Lambda_\nu$, this means $\alpha^{2m} = \nu(2m\mu)^{-1}$. Since $\nu(2m\mu)=\nu(m\mu)^2$, the result follows.
\end{proof}

We now define the logical system and the encoding. Under the Bargmann transform of Section \ref{73mes5iz}, the code space $\code_{\Lambda,\nu}$ is identified with the space of canonical theta functions $\Theta_{\Lambda,\nu}$. In particular, it is the Stone--von Neumann representation of $\Heis_{\Lambda,\nu}$, with the hermitian inner product transported from $\Theta_{\Lambda,\nu}$. Thus, after choosing an isomorphism between $\Heis_{\Lambda,\nu}$ and the standard Heisenberg group $\Heis_D$, the code space can be identified with the standard qudit Hilbert space.

\begin{definition}[Encodings]
Let $(\Lambda,\nu)$ be a GKP code of type $D = \diag(d_1,\ldots,d_n)$. The \defn{logical space} of $(\Lambda,\nu)$ is
\[
\C^{\otimes D} \coloneqq \C^{d_1} \otimes \cdots \otimes \C^{d_n}.
\]
An \defn{encoding} is a unitary isomorphism
\[
\mathrm{Enc} : \C^{\otimes D} \overset{\cong}{\too} \code_{\Lambda,\nu}
\]
such that the induced map
\begin{equation}\label{9puj0mpl}
\U(\C^{\otimes D}) \too \U(\code_{\Lambda,\nu}), \quad U \mtoo \mathrm{Enc} \circ U \circ \mathrm{Enc}^{-1}
\end{equation}
restricts to an isomorphism $\Heis_D \cong \Heis_{\Lambda,\nu}$.
\end{definition}

Encodings exist by Proposition \ref{3ujxcjsa}. Moreover, if $\mathrm{Enc}$ is an encoding, then \eqref{9puj0mpl} automatically restricts to isomorphisms
\[
\Pauli_D \cong \Pauli_{\Lambda,\nu}
\quad\text{and}\quad
\Cliff_D \cong \Cliff_{\Lambda,\nu}.
\]
More explicitly, an encoding is obtained by choosing an isomorphism $\Heis_D \cong \Heis_{\Lambda,\nu}$ and then choosing the corresponding Stone--von Neumann intertwiner $\C^{\otimes D} \cong \code_{\Lambda,\nu}$. This intertwiner is unique up to a global phase once the Heisenberg group isomorphism has been fixed.

Thus a GKP code $(\Lambda,\nu)$ together with an encoding gives a logical space $\mathscr{H}_{\mathrm{logical}} = \C^{\otimes D}$, an ideal physical space $\mathscr{H}_{\mathrm{physical}} = \S'(\R^n)$, a code space $\mathscr{H}_{\mathrm{code}} = \code_{\Lambda,\nu}$, and a linear embedding
\[
\mathrm{Enc} : \C^{\otimes D} \too \S'(\R^n).
\]
The asymptotic form of the isometry requirement will be proved in Section \ref{b9j929je}.

The Pauli gates on $\C^{\otimes D}$ can be implemented physically by translations $\alpha T_\mu$ with $\mu \in \Lambda^\perp$ and $\alpha^{2m} = \nu(m\mu)^{-2}$. These operators preserve the code space $\code_{\Lambda,\nu}$ and induce the Pauli gates under the encoding. In Section \ref{5nk36oj5}, we will show that this implementation extends from Pauli gates to all Clifford gates.

\subsection{GKP codes as complex abelian varieties}

We now give the geometric interpretation of GKP codes in the language of complex abelian varieties.

Recall that if $(X, L)$ is a polarized complex abelian variety of type $D = \diag(d_1, \ldots, d_n)$, a \defn{theta structure} on $(X, L)$ is an isomorphism of $\G_c(L) \cong \Heis_D$ restricting to the identity on $\U(1)$ (see e.g.\ \cite[Section 6.6]{BirkenhakeLange2004ComplexAbelianVarieties}). 

The following theorem summarizes the previous sections of this paper.

\begin{theorem}\label{p2fyb03t}
There is a one-to-one correspondence between GKP codes $(\Lambda,\nu)$ and pairs $(X,L)$ consisting of a complex abelian variety $X$ together with an ample line bundle $L$. Under this correspondence, $X = V/\Lambda$ and $L = L(H,\nu)$ is the line bundle obtained from $\nu$ by the Appell--Humbert theorem. Moreover, the Bargmann transform gives an isomorphism
\begin{equation}\label{8bx3kuxn}
\code_{\Lambda, \nu} \cong H^0(X, L),
\end{equation}
from the code space to the space of canonical theta functions. Under this isomorphism, the Heisenberg group of $(\Lambda,\nu)$ is identified with the compact theta group of $(X,L)$,
\[
\Heis_{\Lambda,\nu} \cong \G_c(L).
\]
Finally, under this correspondence, encodings of the GKP code $(\Lambda,\nu)$ are equivalent to theta structures on $(X,L)$.
\end{theorem}

\begin{proof}
Let $(\Lambda,\nu)$ be a GKP code. As recalled in Section \ref{23yhqu4v}, the lattice $\Lambda$ defines a complex abelian variety $X = V/\Lambda$, and the semicharacter $\nu$ defines an ample line bundle $L = L(H,\nu)$ on $X$ by the Appell--Humbert theorem. Conversely, after choosing the universal cover $V \to X$, the Appell--Humbert description of $L$ gives a hermitian form $H$ and a semicharacter $\nu$ on the kernel $\Lambda$ of the covering map. This recovers the GKP code $(\Lambda,\nu)$ and gives the stated correspondence.

The Bargmann transform identifies $\code_{\Lambda,\nu}$ with the space $\Theta_{\Lambda,\nu}$ of canonical theta functions. By Section \ref{23yhqu4v}, this space is canonically identified with $H^0(X,L)$. This gives the isomorphism \eqref{8bx3kuxn}. The same section identifies the Heisenberg group $\Heis_{\Lambda,\nu}$ with the compact theta group $\G_c(L)$, compatibly with the representations on $\Theta_{\Lambda,\nu}$ and $H^0(X,L)$.

It remains to compare encodings with theta structures. Under the isomorphism $\Heis_{\Lambda,\nu} \cong \G_c(L)$, a theta structure $\G_c(L) \cong \Heis_D$ is equivalently an isomorphism of central extensions
\[
\Heis_D \cong \Heis_{\Lambda,\nu}
\]
restricting to the identity on $\U(1)$. By Proposition \ref{3ujxcjsa}, such an isomorphism determines, uniquely up to a global phase, the corresponding unitary identification
\[
\C^{\otimes D} \cong \code_{\Lambda,\nu}
\]
of Stone--von Neumann representations. This is exactly an encoding in the sense defined above. Conversely, every encoding arises in this way from its underlying isomorphism $\Heis_D \cong \Heis_{\Lambda,\nu}$, and hence from a theta structure.
\end{proof}

\section{Asymptotic isometry of encodings}\label{b9j929je}

Let $(\Lambda,\nu)$ be a GKP code of type $D = \diag(d_1,\ldots,d_n)$, and let
\[
\mathrm{Enc} : \C^{\otimes D} \too \code_{\Lambda,\nu} \subset \S'(\R^n)
\]
be an encoding. The goal of this section is to show that, although this ideal encoding does not take values in a Hilbert space, it can be approximated by square-integrable encodings that are asymptotically isometric after a scalar renormalization of the $L^2$ inner product.

Following \cite{royer2022encoding}, consider the Gaussian envelope operators
\[
E_\beta : \S'(\R^n) \too \S'(\R^n), \quad E_\beta \coloneqq \exp(-\beta \sum_{j=1}^n \hat{n}_j),
\]
for $\beta > 0$, where $\hat{n}_j = \hat{a}_j^\dagger \hat{a}_j$.
The next proposition shows that these operators regularize tempered distributions into square-integrable functions.

\begin{proposition}\label{18d8ov7y}
For every tempered distribution $\psi \in \S'(\R^n)$ and every $\beta > 0$, we have $E_\beta\psi \in L^2(\R^n)$.
\end{proposition}

\begin{proof}
Identify $V$ with $\C^n$ as in Section \ref{73mes5iz}, and use the Bargmann transform to identify $L^2(\R^n)$ with $\mathcal{F}(\C^n)$ and $\S'(\R^n)$ with $\E(\C^n)$. By \eqref{5ctxjsq8}, the operator $E_\beta$ is the rescaling
\[
(E_\beta f)(z) = f(e^{-\beta}z)
\]
on $\E(\C^n)$.

Let $f \in \E(\C^n)$. By the growth condition \eqref{w19o40qe}, there exist constants $k \ge 0$ and $C > 0$ such that
\[
|f(z)| \le C e^{\pi |z|^2/2}(1 + |z|^2)^{k/2}
\]
for all $z \in \C^n$. Hence
\[
|f(e^{-\beta}z)|^2 e^{-\pi |z|^2}
\le
C^2(1 + e^{-2\beta}|z|^2)^k e^{-\pi(1-e^{-2\beta})|z|^2}.
\]
Since $1-e^{-2\beta} > 0$, the right-hand side is integrable on $\C^n$. Thus $E_\beta f \in \mathcal{F}(\C^n)$, and therefore $E_\beta\psi \in L^2(\R^n)$.
\end{proof}

For each $\beta \ge 0$, define the deformed encoding
\begin{equation}\label{04opl8tt}
\mathrm{Enc}_\beta \coloneqq E_\beta \circ \mathrm{Enc} : \C^{\otimes D} \too \mathcal{S}'(\R^n).
\end{equation}
By Proposition \ref{18d8ov7y}, this takes values in $L^2(\R^n)$ for all $\beta > 0$. Moreover, the deformed encodings converge to the ideal encoding in the sense that
\[
\lim_{\beta \to 0} \mathrm{Enc}_\beta(\varphi) = \mathrm{Enc}(\varphi)
\]
in $\S'(\R^n)$ for every $\varphi \in \C^{\otimes D}$.

The next theorem shows that, after rescaling the $L^2$ inner product by a scalar depending on $\beta$, the deformed encodings become asymptotically isometric.

\begin{theorem}\label{hr1j1x4z}
There exists $a > 0$ such that
\[
\ip{\varphi, \psi}
=
c(\beta)\ip{\operatorname{Enc}_\beta(\varphi), \operatorname{Enc}_\beta(\psi)}_{L^2}
+
O(e^{-a/\beta})\|\varphi\|\|\psi\|, \quad \text{ as } \beta \to 0,
\]
for all $\varphi, \psi \in \C^{\otimes D}$, where
\[
c(\beta) \coloneqq (1 - e^{-2\beta})^n\det(\Lambda).
\]
In particular,
\[
\ip{\varphi, \psi}
=
\lim_{\beta \to 0}
\det(\Lambda)(2\beta)^n
\ip{\operatorname{Enc}_\beta(\varphi), \operatorname{Enc}_\beta(\psi)}_{L^2}.
\]
\end{theorem}

\begin{proof}
We use the Bargmann transform to identify $\S'(\R^n)$ with $\E(\C^n)$ and $L^2(\R^n)$ with $\mathcal{F}(\C^n)$, as in the proof of Proposition \ref{18d8ov7y}. Under this identification, $\code_{\Lambda,\nu}$ is identified with $\Theta_{\Lambda,\nu}$, equipped with the inner product
\[
\ip{f, g}_{\Theta}
=
\frac{1}{\pi^n}
\int_{\C^n/\Lambda} f(z)\overline{g(z)}e^{-\pi|z|^2}\,dz\,d\bar{z};
\]
see \eqref{cu7nnb56}.
Thus it suffices to prove that
\[
c(\beta)\ip{E_\beta f, E_\beta g}_{\mathcal{F}}
=
\ip{f, g}_{\Theta}
+
O(e^{-a/\beta})\|f\|_{\Theta}\|g\|_{\Theta}
\]
for all $f,g \in \Theta_{\Lambda,\nu}$.

Let $F \subset \C^n$ be a fundamental domain for $\Lambda$. Since $(E_\beta f)(z)=f(e^{-\beta}z)$, we have
\begin{align*}
\ip{E_\beta f, E_\beta g}_{\mathcal{F}}
&=
\frac{e^{2n\beta}}{\pi^n}
\int_{\C^n} f(z)\overline{g(z)}e^{-\pi e^{2\beta}|z|^2}\,dz\,d\bar{z} \\
&=
\frac{e^{2n\beta}}{\pi^n}
\sum_{\lambda \in \Lambda}
\int_F f(z+\lambda)\overline{g(z+\lambda)}
e^{-\pi e^{2\beta}|z+\lambda|^2}\,dz\,d\bar{z}.
\end{align*}
Using the automorphy relation for theta functions, the last expression becomes
\[
\frac{e^{2n\beta}}{\pi^n}
\int_{\C^n/\Lambda}
f(z)\overline{g(z)}e^{-\pi|z|^2}K_\beta(z)\,dz\,d\bar{z},
\]
where
\[
K_\beta(z)
\coloneqq
\sum_{\lambda \in \Lambda}
e^{-\pi\alpha|z+\lambda|^2},
\qquad
\alpha \coloneqq e^{2\beta}-1.
\]
By the Poisson summation formula,
\[
K_\beta(z)
=
\frac{1}{\alpha^n\det(\Lambda)}
\sum_{\xi \in \Lambda^*}
e^{-\frac{\pi}{\alpha}|\xi|^2}
e^{2\pi i\operatorname{Re}(\bar{\xi}\cdot z)}.
\]
Therefore
\[
c(\beta)\ip{E_\beta f, E_\beta g}_{\mathcal{F}}
=
\ip{f, g}_{\Theta}
+
R_\beta(f,g),
\]
with
\[
c(\beta)
=
\frac{\alpha^n\det(\Lambda)}{e^{2n\beta}}
=
(1-e^{-2\beta})^n\det(\Lambda),
\]
and
\[
|R_\beta(f,g)|
\le
A_\beta\|f\|_{\Theta}\|g\|_{\Theta},
\qquad
A_\beta
\coloneqq
\sum_{\xi \in \Lambda^*\setminus\{0\}}
e^{-\frac{\pi}{\alpha}|\xi|^2}.
\]

It remains to bound $A_\beta$. Let
\[
\lambda_1 \coloneqq \min\{|\xi| : \xi \in \Lambda^*,\, \xi \ne 0\}.
\]
For $\beta$ sufficiently small, $\alpha \le 2$, and hence
\[
A_\beta
\le
e^{-\frac{\pi}{\alpha}\lambda_1^2}
\sum_{\xi \in \Lambda^*\setminus\{0\}}
e^{-\frac{\pi}{2}(|\xi|^2-\lambda_1^2)}
\le
C_\Lambda e^{-\frac{\pi}{\alpha}\lambda_1^2}
\]
for some constant $C_\Lambda > 0$. Since $\alpha = e^{2\beta}-1 \sim 2\beta$ as $\beta \to 0$, there exists $a > 0$ such that
\[
A_\beta = O(e^{-a/\beta}).
\]
This proves the stated asymptotic formula. The final limit follows from
\[
c(\beta) \sim \det(\Lambda)(2\beta)^n
\]
as $\beta \to 0$.
\end{proof}

Thus the logical inner product is recovered from the renormalized $L^2$ inner products of the finite-energy approximations. This proves the asymptotic form of the isometry requirement \ref{0n07smer}.

\section{Clifford gates via Gaussian unitaries}\label{5nk36oj5}

Let $(\Lambda,\nu)$ be a GKP code of type $D$ with encoding
\[
\mathrm{Enc} : \C^{\otimes D} \overset{\cong}{\too} \code_{\Lambda,\nu}.
\]
In this section we prove the logical gate implementation property \ref{adfuq8hq} for Clifford gates. More precisely, we show that every Clifford gate on $\C^{\otimes D}$ is induced by a \defn{Gaussian unitary}, i.e.\ an operator in the extended metaplectic representation
\[
\rho : \Mp(V) \ltimes \HH(V) \too \GL(\S'(\R^n));
\]
see Section \ref{zz4171u8}.

We also study an important subclass of these operators. Recall that the compatible complex structure on $V$ determines the unitary group
\[
\U(V) = \Sp(V) \cap \mathrm{O}(V),
\]
which embeds in $\Mp(V)$ through the canonical splitting of $\pi : \pi^{-1}(\U(V)) \to \U(V)$. A Gaussian unitary is called \defn{passive} if it comes from the subgroup $\U(V) \ltimes \HH(V)$.\footnote{This is a slight abuse of terminology; see Footnote~\ref{f3rsimhu}.} In quantum optics, these are the Gaussian unitaries that do not involve squeezing; they are built from displacements and unitary changes of modes, such as phase shifts and beam splitters (see, e.g., \cite{weedbrook2012gaussian}). They are therefore typically easier to implement physically than general Gaussian unitaries. We will show that the group of passive Clifford gates is naturally isomorphic to a quotient of the group of automorphisms of the corresponding polarized abelian variety.

The first step is to identify the Gaussian unitaries that preserve the code space. Since $\code_{\Lambda,\nu}$ is the simultaneous $+1$ eigenspace of $\Lambda_\nu$, this amounts to considering the normalizer of $\Lambda_\nu$ in the extended metaplectic group.

\begin{definition}\label{1m8hfr2m}
The \defn{normalizer group} of a GKP code $(\Lambda,\nu)$ is the quotient
\[
\mathcal{N}_{\Lambda,\nu}
\coloneqq
N_{\Mp(V) \ltimes \HH(V)}(\Lambda_\nu)/\Lambda_\nu.
\]
The \defn{passive normalizer group} is the subgroup
\[
\mathcal{N}_{\Lambda,\nu}^{\passive}
\coloneqq
N_{\U(V) \ltimes \HH(V)}(\Lambda_\nu)/\Lambda_\nu.
\]
\end{definition}

Every element of $N_{\Mp(V) \ltimes \HH(V)}(\Lambda_\nu)$ preserves the simultaneous eigenspace of $\Lambda_\nu$, and hence preserves $\code_{\Lambda,\nu}$. Since $\Lambda_\nu$ acts trivially on $\code_{\Lambda,\nu}$, this gives a representation
\begin{equation}\label{1jxc4iv1}
\mathcal{N}_{\Lambda,\nu} \too \GL(\code_{\Lambda,\nu}).
\end{equation}
Our first main result identifies the image of this representation.

\begin{theorem}\label{cczuu62j}
The representation \eqref{1jxc4iv1} is unitary. Moreover, its image is precisely the Clifford group $\Cliff_{\Lambda,\nu}$, giving a surjection
\begin{equation}\label{j03xt30z}
\mathcal{N}_{\Lambda,\nu} \too \Cliff_{\Lambda,\nu}.
\end{equation}
In particular, every Clifford gate on $\mathscr{H}_{\mathrm{logical}} = \C^{\otimes D}$ can be realized by a Gaussian unitary on $\mathscr{H}_{\mathrm{physical}} = \S'(\R^n)$.
\end{theorem}

To prove Theorem \ref{cczuu62j}, we first give a more explicit description of the normalizer group. Consider the groups
\[
\Sp(\Lambda) \coloneqq \{M \in \Sp(V) : M\Lambda = \Lambda\}
\quad\text{and}\quad
\Mp(\Lambda) \coloneqq \{M \in \Mp(V) : M\Lambda = \Lambda\}.
\]
For $M \in \Mp(\Lambda)$, the map
\[
\Lambda \too \U(1), \quad \lambda \mtoo \frac{\nu(M\lambda)}{\nu(\lambda)}
\]
is a character. Hence there is a unique class $u_M \in V/\Lambda^\perp$ such that
\begin{equation}\label{g7nmwq61}
\frac{\nu(M\lambda)}{\nu(\lambda)}
=
e^{2\pi iE(u_M,M\lambda)}
\end{equation}
for all $\lambda \in \Lambda$.

\begin{proposition}\label{utnj2ijc}
We have
\[
N_{\Mp(V) \ltimes \HH(V)}(\Lambda_\nu)
=
\{(M,u_M+\mu,\alpha) : M \in \Mp(\Lambda),\ \mu \in \Lambda^\perp,\ \alpha \in \U(1)\}.
\]
In particular, the sequence
\[
\begin{tikzcd}
1 \arrow{r} & \Heis_{\Lambda,\nu} \arrow{r} & \mathcal{N}_{\Lambda,\nu} \arrow{r} & \Mp(\Lambda) \arrow{r} & 1
\end{tikzcd}
\]
is exact.
\end{proposition}

\begin{proof}
Let $(M,u,\alpha) \in N_{\Mp(V) \ltimes \HH(V)}(\Lambda_\nu)$. Then, for all $\lambda \in \Lambda$, we have
\begin{equation}\label{gd3ubnta}
(M,u,\alpha)(1,\lambda,\nu(\lambda)^{-1})(M,u,\alpha)^{-1}
=
(1,M\lambda,\nu(\lambda)^{-1}e^{-2\pi iE(u,M\lambda)})
\in \Lambda_\nu.
\end{equation}
It follows that $M\Lambda \subset \Lambda$. Applying the same argument to $(M,u,\alpha)^{-1}$ gives $M^{-1}\Lambda \subset \Lambda$, and hence $M \in \Mp(\Lambda)$. By \eqref{gd3ubnta}, we also have
\[
\frac{\nu(M\lambda)}{\nu(\lambda)}
=
e^{2\pi iE(u,M\lambda)}
\]
for all $\lambda \in \Lambda$. By the definition of $u_M$, this means that $u = u_M+\mu$ for some $\mu \in \Lambda^\perp$.

Conversely, if $M \in \Mp(\Lambda)$, $\mu \in \Lambda^\perp$, and $\alpha \in \U(1)$, then \eqref{g7nmwq61} shows that $(M,u_M+\mu,\alpha)$ normalizes $\Lambda_\nu$. This proves the stated description of the normalizer. The exact sequence follows by quotienting by $\Lambda_\nu$.
\end{proof}

\begin{proof}[Proof of Theorem \ref{cczuu62j}]
We first show that the representation of $\mathcal{N}_{\Lambda, \nu}$ on $\code_{\Lambda, \nu}$ is unitary.

For every $U \in \mathcal{N}_{\Lambda, \nu}$, the composition of the representation $\rho : \Heis_{\Lambda, \nu} \to \GL(\code_{\Lambda, \nu})$ with conjugation by $U$ is an irreducible representation of $\Heis_{\Lambda, \nu}$. Moreover, this representation is unitary with respect to the hermitian inner product
\[
\ip{\varphi, \psi}_U \coloneqq \ip{U\varphi, U\psi}.
\]
By the Stone--von Neumann theorem and Schur's lemma, there is a constant $c(U) > 0$ such that
\[
\ip{U\varphi, U \psi} = c(U) \ip{\varphi, \psi}
\]
for all $\varphi, \psi \in \code_{\Lambda, \nu}$.
Moreover, the map
\[
c : \mathcal{N}_{\Lambda, \nu} \too \R_{>0}
\]
is a group homomorphism. We need to show that $c$ is trivial. Since $\Heis_{\Lambda, \nu}$ acts unitarily, we have $c(\Heis_{\Lambda, \nu}) = 1$, so $c$ descends to a homomorphism
\[
c : \mathcal{N}_{\Lambda, \nu}/\Heis_{\Lambda, \nu} \too \R_{>0}.
\]
By Proposition \ref{utnj2ijc}, this quotient is $\Mp(\Lambda)$. Since $\ker(\Mp(\Lambda) \to \Sp(\Lambda)) = \{\pm1\}$ is finite, $c$ further descends to a homomorphism
\[
c : \Sp(\Lambda) \too \R_{>0}.
\]
It then suffices to show that $\Sp(\Lambda)$ has finite abelianization. For the case $n = 1$, we have $\Sp(\Lambda) \cong \Sp(2,\Z)$ and the abelianization $\Sp(2,\Z)^{\mathrm{ab}} = \Z/12\Z$ is finite; see e.g.\ \cite[Theorem 3.8]{ConradSL2Z}. For $n \ge 2$, this follows from \cite[Proposition 6.19(iii)]{margulis1991discrete}; indeed, $\Sp(\Lambda)$ is an arithmetic subgroup of $\Sp(2n,\R)$, so it is a lattice in $\Sp(2n,\R)$ by the Borel--Harish-Chandra theorem.

It follows that the image of \eqref{1jxc4iv1} is contained in the Clifford group $\Cliff_{\Lambda, \nu}$. It remains to show that the restriction $\mathcal{N}_{\Lambda, \nu} \to \Cliff_{\Lambda, \nu}$ is surjective.
Consider the commutative diagram
\[
\begin{tikzcd}
1 \arrow{r} & \Heis_{\Lambda, \nu} \arrow{r} \arrow[equals]{d} & \mathcal{N}_{\Lambda, \nu} \arrow{r}\arrow{d} & \Mp(\Lambda) \arrow{r} \arrow{d} & 1 \\
1 \arrow{r} & \Heis_{\Lambda, \nu} \arrow{r} & \Cliff_{\Lambda, \nu} \arrow{r} & \Sp(\Lambda^\perp/\Lambda) \arrow{r} & 1.
\end{tikzcd}
\]
By Proposition \ref{utnj2ijc} and Proposition \ref{xor7ubv5}, the two rows are exact. By \cite[Theorem 2]{brasch1993lifting}, the third vertical map is surjective. It follows that the second vertical map is also surjective.
\end{proof}

Theorem \ref{cczuu62j} allows us to define the passive Clifford gates.

\begin{definition}
The \defn{passive Clifford group}, denoted $\Cliff_{\Lambda,\nu}^{\passive}$, is the image of $\mathcal{N}_{\Lambda,\nu}^{\passive}$ in $\Cliff_{\Lambda,\nu}$ under the surjection \eqref{j03xt30z}. If an encoding identifies $\Cliff_{\Lambda,\nu}$ with the standard Clifford group $\Cliff_D$, we denote the corresponding subgroup of $\Cliff_D$ by $\Cliff_D^{\passive}$.
\end{definition}

We now give a geometric interpretation of the passive Clifford group. Let $(X,L)$ be the polarized complex abelian variety associated with the GKP code $(\Lambda,\nu)$. Let $\widetilde{\Aut}(X,L)$ denote the group of automorphisms of $L$ preserving the hermitian fibre metric, i.e.\ the group of line bundle isomorphisms $\varphi : L \to L$ covering a biholomorphism $f : X \to X$ and restricting to hermitian isometries on the fibres. In particular, $\widetilde{\Aut}(X,L)$ contains the compact theta group $\G_c(L)$ as the subgroup of automorphisms covering translations. 

Recall that every biholomorphism $f : X \to X$ decomposes as $f = t_x \circ g$, where $g$ fixes the origin. There is then a short exact sequence
\[
1 \too \G_c(L) \too \widetilde{\Aut}(X,L) \too \Aut(X,L) \too 1,
\]
where $\Aut(X,L)$ is the group of biholomorphisms $g : X \to X$ fixing the origin such that $c_1(g^*L) = c_1(L)$.

The group $\widetilde{\Aut}(X,L)$ acts unitarily on $H^0(X,L)$ as follows. If $\varphi \in \widetilde{\Aut}(X,L)$ covers $f : X \to X$, then
\[
(\varphi \cdot \sigma)(x) = \varphi(\sigma(f^{-1}(x)))
\]
for all $\sigma \in H^0(X,L)$ and $x \in X$.

We now relate $\widetilde{\Aut}(X,L)$ to the passive normalizer group. Let
\[
\U(\Lambda) \coloneqq \{M \in \U(V) : M\Lambda = \Lambda\}.
\]
By Proposition \ref{utnj2ijc}, restricted to $\U(V)\ltimes\HH(V)$, the normalizer is
\[
N_{\U(V)\ltimes\HH(V)}(\Lambda_\nu)
=
\{(M,u_M+\mu,\alpha) : M \in \U(\Lambda),\ \mu \in \Lambda^\perp,\ \alpha \in \U(1)\}.
\]
For $(M,w,\alpha) \in N_{\U(V)\ltimes\HH(V)}(\Lambda_\nu)$, define
\[
\varphi_{(M,w,\alpha)} : L \too L
\]
by
\[
\varphi_{(M,w,\alpha)}([v,z])
=
[Mv-w,\alpha e^{\pi(\frac{1}{2}H(w,w)-H(Mv,w))}z].
\]
This is a well-defined element of $\widetilde{\Aut}(X,L)$, and it is trivial when $(M,w,\alpha) \in \Lambda_\nu$.

\begin{theorem}\label{s14jrrx8}
The sequence
\[
1 \too \Lambda_\nu \too N_{\U(V)\ltimes\HH(V)}(\Lambda_\nu) \too \widetilde{\Aut}(X,L) \too 1
\]
is exact. Hence
\[
\mathcal{N}_{\Lambda,\nu}^{\passive} \cong \widetilde{\Aut}(X,L).
\]
Moreover, under the Bargmann transform, the representation of $\mathcal{N}_{\Lambda,\nu}^{\passive}$ on $\code_{\Lambda,\nu}$ corresponds to the natural representation of $\widetilde{\Aut}(X,L)$ on $H^0(X,L)$.
\end{theorem}

\begin{proof}
Let $\varphi \in \widetilde{\Aut}(X,L)$ cover $f : X \to X$. Then $f$ can be written as
\[
f([v]) = [Mv - w]
\]
for some $M \in \GL(V)$ and $w \in V$. Since $\varphi$ is an isomorphism $L \cong f^*L$, and since $L = L(H,\nu)$, the Appell--Humbert theorem gives
\[
M^*H = H
\]
and
\[
\nu(\lambda) = \nu(M\lambda)e^{-2\pi iE(w,M\lambda)}
\]
for all $\lambda \in \Lambda$. Thus $M \in \U(V)$, $M\Lambda = \Lambda$, and, in the notation of \eqref{g7nmwq61}, we have $w = u_M$ modulo $\Lambda^\perp$.

We claim that there exists $\alpha \in \U(1)$ such that
\begin{equation}\label{6u39ewc1}
\varphi([v,z])
=
[Mv-w,\alpha e^{\pi(\frac{1}{2}H(w,w)-H(Mv,w))}z]
\end{equation}
for all $[v,z] \in L$. Indeed, after factoring out the displayed multiplier, $\varphi$ is given by a holomorphic function $b : V \to \C^*$ of the form
\[
\varphi([v,z])
=
[Mv-w,b(v)e^{\pi(\frac{1}{2}H(w,w)-H(Mv,w))}z].
\]
The condition that this descends to the quotient by $\Lambda$ implies that $b$ is $\Lambda$-periodic. Hence $b$ descends to a holomorphic function on the compact complex torus $X$, and is therefore constant. Since $\varphi$ preserves the hermitian fibre metric, this constant lies in $\U(1)$.

It follows that every element of $\widetilde{\Aut}(X,L)$ is of the form $\varphi_{(M,w,\alpha)}$ for some $(M,w,\alpha) \in N_{\U(V)\ltimes\HH(V)}(\Lambda_\nu)$. Conversely, the formula defining $\varphi_{(M,w,\alpha)}$ gives an element of $\widetilde{\Aut}(X,L)$ for every $(M,w,\alpha)$ in the passive normalizer. Thus we have a surjective homomorphism
\[
N_{\U(V)\ltimes\HH(V)}(\Lambda_\nu) \too \widetilde{\Aut}(X,L).
\]
Its kernel is precisely $\Lambda_\nu$, since $\varphi_{(M,w,\alpha)}$ is the identity if and only if $M = 1$, $w \in \Lambda$, and $\alpha = \nu(w)^{-1}$. This proves the exactness of the sequence and hence the isomorphism.

Finally, the compatibility of the representations follows directly from \eqref{kilzod4v}.
\end{proof}

We can now describe the passive Clifford group geometrically. Recall that $\Aut(X, L)$ preserves $K(L) = \Lambda^\perp/\Lambda \subset X$ and the commutator map, giving a group homomorphism
\[
\Phi_{X, L}: \Aut(X, L) \too \Sp(K(L)).
\]

\begin{theorem}\label{fi2al08e}
Let $(X, L)$ be a polarized complex abelian variety of type $D = \diag(d_1,\ldots,d_n)$ with theta structure. Then there is a group isomorphism
\begin{equation}\label{92ijm96l}
\im(\Phi_{X, L}) \cong \Cliff_D^{\passive}/\Heis_D.
\end{equation}
If $d_1 \ge 3$, then $\Phi_{X, L}$ is injective, so
\[
\Aut(X, L) \cong \Cliff_D^\passive/\Heis_D.
\]
\end{theorem}

\begin{proof}
Let
\[
\Sp(K)^\passive \coloneqq \im(\Cliff_{\Lambda, \nu}^\passive \too \Sp(K)).
\]
By Theorem \ref{s14jrrx8} and Proposition \ref{xor7ubv5}, we have a morphism of short exact sequences
\[
\begin{tikzcd}
1 \arrow{r} & \G_c(L) \arrow{r} \arrow{d}{\cong} & \widetilde{\Aut}(X, L) \arrow{r} \arrow{d} & \Aut(X, L) \arrow{r} \arrow{d} & 1 \\
1 \arrow{r} & \Heis_{\Lambda, \nu} \arrow{r} & \Cliff_{\Lambda, \nu}^\passive \arrow{r} & \Sp(K)^\passive \arrow{r} & 1.
\end{tikzcd}
\]
By Theorem \ref{cczuu62j}, the second vertical map is surjective, and hence so is the third. Then the isomorphism \eqref{92ijm96l} follows by identifying $\Heis_{\Lambda, \nu} \cong \Heis_D$ and $\Cliff_{\Lambda, \nu} \cong \Cliff_D$.

For a positive integer $k$, set $X_k \coloneqq \{x \in X : k x = 0\}$. Then for all $k \ge 3$,  the restriction map $\Aut(X, L) \to \Aut(X_k)$ is injective \cite[Corollary 5.1.10]{BirkenhakeLange2004ComplexAbelianVarieties}. Since $d_1$ is the smallest elementary divisor, we have $X_{d_1} \subset K(L)$. It follows that $\Phi_{X, L}$ is injective when $d_1 \ge 3$.
\end{proof}

Thus the passive Clifford gates of a GKP code have a purely geometric description. The Pauli gates come from the compact theta group $\G_c(L)$, while the additional passive Clifford gates are controlled by the image of the automorphism group $\Aut(X,L)$ in $\Sp(K)$. Consequently, for a fixed polarization type, the problem of producing GKP codes with many readily implementable passive Clifford gates becomes the problem of finding polarized complex abelian varieties whose image $\im(\Aut(X,L) \to \Sp(K))$ is large.

\begin{corollary}\label{zdzp6h79}
Let $C$ be a Riemann surface of genus $g \ge 2$ and let $(X, L)$ be the Jacobian of $C$ with its canonical principal polarization $L$. For all $d \ge 3$ and encoding $\C^{\otimes D} \cong H^0(X, L^d)$, where $D = \diag(d, \ldots, d)$, we have
\[
\Cliff_D^\passive/\Heis_D \cong \begin{cases} \Aut(C) & \text{ if $C$ is hyperelliptic} \\ \{\pm1\} \times \Aut(C) & \text{ if $C$ is non-hyperelliptic}.
\end{cases}
\]
\end{corollary}

\begin{proof}
By Theorem \ref{fi2al08e}, $\Cliff_D^\passive/\Heis_D \cong \Aut(X, L^d) = \Aut(X, L)$.
The result then follows by Serre's ``precise form'' of the Torelli theorem \cite[Appendix]{LauterSerre2001}.
\end{proof}

It is then natural to look for examples among Jacobians of curves with large automorphism groups. The Klein quartic provides a particularly useful example, yielding a three-mode family of GKP codes with many non-Pauli passive Clifford gates.

\begin{proposition}\label{dhnyif6v}
Let $(X, L)$ be the Jacobian of the Klein quartic with its canonical principal polarization. For any $d \ge 2$ and encoding $\C^d \otimes \C^d \otimes \C^d \cong H^0(X, L^d)$, we have
\[
|\Cliff_D^\passive/\Heis_D| = 
\begin{cases} 
168 & \text{ if } d = 2 \\
336 & \text{ if } d \ge 3.
\end{cases}
\]
\end{proposition}

\begin{proof}
The curve $C$ is a Hurwitz curve of genus $g = 3$, i.e.\ it attains the Hurwitz bound $|{\Aut(C)}| = 84(g - 1) = 168$ \cite{elkies1999klein,levy2001eightfold}.
Since $C$ is non-hyperelliptic, the case $d \ge 3$ follows by Corollary \ref{zdzp6h79}.

Suppose now that $d = 2$. Since $L$ is principal, the map $\phi_L : X \to \widehat{X}$ is an isomorphism. Moreover, $\phi_{L^2} = 2\phi_L$, so $K(L^2) = X_2$. It follows that $-\mathrm{Id}_X$ acts trivially on $K(L^2)$, so $-\mathrm{Id}_X \in \ker(\Phi_{X, L^2})$. By Torelli's theorem \cite[Appendix]{LauterSerre2001}, $\Aut(X, L)/\ip{\pm\mathrm{Id}_X} \cong \Aut(C)$, so $\Phi_{X, L^2}$ descends to $\Aut(C) \to \Sp(K(L^2))$. Since $\Aut(C) \cong \operatorname{PSL}(2, 7)$ is simple \cite{elkies1999klein,levy2001eightfold}, it suffices to show that there exists $f \in \Aut(X, L)$ acting non-trivially on $K(L^2)$. By \cite[Proposition 2.1(ii)]{MarkushevichMoreau2023}, there exists $f \in \Aut(X, L)$ with exactly 16 fixed points. Since $|X_2| = |(\Z/2\Z)^{2g}| = 64$, this automorphism cannot fix $K(L^2)=X_2$ pointwise. It follows that $|{\im(\Phi_{X,L^2})}| = |{\Aut(C)}| = 168$.
\end{proof}

\section{Quantum error correction}\label{k3dqdqez}

We now turn to the error correction property \ref{bd6ih4k1}. The preceding sections described the structure of the code space and the implementation of logical gates. The remaining question is how the code responds to noise. For GKP codes, the relevant errors are translation operators $T_v$, where $v \in V$ is sampled from a probability distribution.

Let $(\Lambda,\nu)$ be a GKP code and let $\psi \in \code_{\Lambda,\nu}$ be a code state. If a displacement error $v \in V$ occurs, the state becomes $T_v\psi$. Measuring the stabilizer group $\Lambda_\nu$ does not reveal $v$ itself, but only its class modulo $\Lambda^\perp$. A decoder must therefore choose a representative of this coset and apply the inverse displacement. If the residual displacement lies in $\Lambda$, the original state is recovered up to phase. If it lies in $\Lambda^\perp \setminus \Lambda$, the correction produces a non-trivial Pauli operator on the code space; this is a logical error.

The purpose of this section is to formulate the decoding problem precisely, first in the language of lattices and probability measures, and then in geometric terms. For a given noise distribution, we consider decoders that maximize the probability of avoiding a logical error; these are the maximum likelihood decoders. This leads to the robustness of a GKP code, defined as the largest possible probability of successful decoding for that noise distribution. We then specialize to centered Gaussian noise of variance $\sigma^2$ and prove that, in the small-variance limit, the logical error probability is governed by the shortest vector in $\Lambda^\perp \setminus \Lambda$. Finally, we reinterpret the resulting notions and estimates in the language of complex abelian varieties.

\subsection{Decoding}

Let $(\Lambda,\nu)$ be a GKP code in $V$. The stabilizer group
\[
S_{\Lambda,\nu} = \{\nu(\lambda)^{-1}T_\lambda : \lambda \in \Lambda\}
\]
is an abelian group of commuting operators on $\S'(\R^n)$. In quantum-mechanical terms, such a commuting family can be measured simultaneously on a state $\psi \in \S'(\R^n)$. Although $S_{\Lambda,\nu}$ is infinite, it is finitely generated, so it suffices to measure a finite generating set. In particular, if $\psi \in \S'(\R^n)$ is a simultaneous eigenstate of $S_{\Lambda,\nu}$, i.e.\
\[
\nu(\lambda)^{-1}T_\lambda \psi = \chi(\lambda)\psi
\]
for all $\lambda \in \Lambda$, where $\chi(\lambda) \in \U(1)$, then the measurement recovers the eigenvalues, or equivalently produces a character
\[
\chi : \Lambda \too \U(1).
\]
Since $\psi$ is already an eigenstate of $S_{\Lambda, \nu}$, it is left unchanged by this measurement.

In particular, if $\psi \in \code_{\Lambda,\nu}$ is a code state, then, by definition, $\psi$ is a simultaneous $+1$ eigenstate of $S_{\Lambda,\nu}$. Thus measuring the stabilizer on $\psi$ gives the trivial character. If a displacement error $v \in V$ has been applied, then the perturbed state $T_v\psi$ is still a simultaneous eigenstate for $S_{\Lambda,\nu}$, but with shifted eigenvalues. Indeed, for all $\lambda \in \Lambda$, we have
\[
\nu(\lambda)^{-1}T_\lambda T_v \psi
=
e^{2\pi iE(v,\lambda)}T_v\psi.
\]
Thus measuring the stabilizer on $T_v\psi$ gives the character
\[
\lambda \mtoo e^{2\pi iE(v,\lambda)}
\]
of $\Lambda$. Equivalently, the measurement determines the coset
\[
v + \Lambda^\perp \in V/\Lambda^\perp.
\]
In other words, the measurement detects the displacement error only modulo $\Lambda^\perp$.
The result $v + \Lambda^\perp$ of this measurement is called the \defn{syndrome}.

The goal of decoding is then to choose a representative $u \in v+\Lambda^\perp$ and apply the inverse displacement to $T_v\psi$. If $\mu \coloneqq v-u \in \Lambda^\perp$, then
\[
T_{-u}T_v\psi = e^{i\theta}T_\mu\psi
\]
for some phase $e^{i\theta}$. Since $\mu \in \Lambda^\perp$, the operator $T_\mu$ preserves the code space and induces a Pauli gate. The correction succeeds precisely when this Pauli gate is trivial, equivalently when $\mu \in \Lambda$. In that case, the corrected state is proportional to $\psi$, which is the appropriate notion of recovery since global phases do not affect quantum states. If $\mu \in \Lambda^\perp \setminus \Lambda$, the correction produces a non-trivial Pauli gate on the code space; this is a \defn{logical error}.

Since changing $u$ by an element of $\Lambda$ only changes the corrected state by a global phase, a decoder only needs to choose an element of $V/\Lambda$ lying over the measured coset in $V/\Lambda^\perp$. Thus decoding is naturally described in terms of the covering map
\[
\phi : V/\Lambda \too V/\Lambda^\perp.
\]

\begin{definition}
A \defn{decoder} is a measurable section of $\phi$, i.e.\ a measurable map
\[
\mathcal{D} : V/\Lambda^\perp \too V/\Lambda
\]
such that
\[
\phi \circ \mathcal{D} = \mathrm{Id}_{V/\Lambda^\perp}.
\]
The set of all decoders is denoted
\[
\mathrm{Dec}_{\Lambda}
\coloneqq
\{\mathcal{D} : V/\Lambda^\perp \to V/\Lambda :
\phi \circ \mathcal{D} = \mathrm{Id}_{V/\Lambda^\perp}
\text{ and }\mathcal{D}\text{ is measurable}\}.
\]
\end{definition}

Given a decoder $\mathcal{D}$ and an error $v \in V$, the correction succeeds precisely when
\[
\mathcal{D}(v+\Lambda^\perp) = v+\Lambda,
\]
or equivalently, when $\pi(v) \in \im \mathcal{D}$, where
\[
\pi : V \too V/\Lambda
\]
is the quotient map.
Indeed, this means that the chosen representative differs from the true error by an element of $\Lambda$, and hence that the residual Pauli gate is trivial.

Let $\rho$ be a Borel probability measure on $V$ describing the distribution of displacement errors. Then the probability of successful decoding using $\mathcal{D}$ is
\[
\rho(\{v \in V : \mathcal{D}(v+\Lambda^\perp) = v+\Lambda\}) = \pi_*\rho(\im \mathcal{D}),
\]
where $\pi_*\rho$ is the pushforward measure of $\rho$ to $V/\Lambda$.

\begin{definition}\label{aszc1bm9}
Let $(\Lambda,\nu)$ be a GKP code in $V$ and let $\rho$ be a Borel probability measure on $V$.
The \defn{robustness} of $(\Lambda,\nu)$ with respect to $\rho$ is
\[
\mathscr{R}_\rho(\Lambda)
\coloneqq
\sup_{\mathcal{D}\in \mathrm{Dec}_{\Lambda}}
\pi_*\rho(\im\mathcal{D}).
\]
A \defn{maximum likelihood decoder}, or \defn{MLD}, is a decoder attaining this supremum.
The \defn{fragility} of $(\Lambda,\nu)$ with respect to $\rho$ is
\[
\mathscr{F}_\rho(\Lambda)
\coloneqq
1-\mathscr{R}_\rho(\Lambda).
\]
\end{definition}

Thus $\mathscr{R}_\rho(\Lambda)$ is the largest possible probability of avoiding a logical error, while $\mathscr{F}_\rho(\Lambda)$ is the corresponding optimal probability of a logical error.

\begin{example}\label{tenixxix}
Take any measure $\rho$ on $V$ such that $\pi_*\rho$ is the Haar measure, i.e.\ the uniform measure on $V/\Lambda$ corresponding to the volume form induced by $H$. We claim that
\begin{equation}\label{173okzhx}
\mathscr{R}_\rho(\Lambda) = \frac{1}{|\Lambda^\perp/\Lambda|}.
\end{equation}
Indeed, let $\mathcal{D}$ be any decoder. The translates $\im\mathcal{D} + k$, for $k \in K = \Lambda^\perp/\Lambda$, are pairwise disjoint and cover $V/\Lambda$. Since $\pi_*\rho$ is the Haar measure, all these translates have the same measure. Therefore
\[
1 = \sum_{k\in K} \pi_*\rho(\im\mathcal{D}+k) = |K|\pi_*\rho(\im\mathcal{D}).
\]
It follows that every decoder has the same success probability $1/|K|$, proving \eqref{173okzhx}.
\end{example}

Example \ref{tenixxix} represents the worst possible case, as shown by the following result.

\begin{proposition}\label{3d9mrs7s}
For every Borel probability measure $\rho$ on $V$, we have
\[
\mathscr{R}_\rho(\Lambda)\ge \frac{1}{|\Lambda^\perp/\Lambda|}.
\]
\end{proposition}

\begin{proof}
Let $\mathcal{D}$ be any decoder. For each $k \in K$, the map $\mathcal{D}_k \coloneqq \mathcal{D} + k$ is another decoder. Therefore, as in Example \ref{tenixxix}, we have
\[
1 = \sum_{k \in K}\pi_*\rho(\im\mathcal{D} + k) = \sum_{k \in K} \pi_*\rho(\im \mathcal{D}_k).
\]
It follows that $\pi_*\rho(\im \mathcal{D}_k) \ge 1/|K|$ for some $k\in K$.
\end{proof}

One way to construct a decoder is to choose a fundamental domain $Q \subset V$
for $\Lambda^\perp$ and define $\mathcal{D}_Q : V/\Lambda^\perp \cong Q \hookrightarrow V \to V/\Lambda$. For instance, if $Q$ is contained in the closed Voronoi cell of $\Lambda^\perp$, then $\mathcal{D}_Q$ is known as the \defn{maximum energy decoder} (MED)~\cite{conrad2022gottesman}. On the other hand, the maximum likelihood decoder is typically not of this form.

We now describe how to construct maximum likelihood decoders when the noise distribution is given by a sufficiently regular density. Suppose that $\rho$ is induced by a probability density function $f : V \to [0,\infty)$ with respect to the volume form on $V$. Assume, for simplicity, that $f$ is a Schwartz function. Then the periodization
\begin{equation}\label{0irhdvp5}
F : V/\Lambda \too \R, \quad
F(v) = \sum_{\lambda \in \Lambda} f(v+\lambda),
\end{equation}
converges absolutely and defines a smooth function on $V/\Lambda$.

The function $F$ is the density of the pushforward measure $\pi_*\rho$ on $V/\Lambda$. Therefore the probability of success of a decoder $\mathcal{D}$ can be computed by integrating $F$ over the image of $\mathcal{D}$. This gives a simple criterion for maximum likelihood decoding: for each measured syndrome $y \in V/\Lambda^\perp$, choose a point in the finite fibre $\phi^{-1}(y)$ where $F$ is maximal.

\begin{proposition}\label{rq2fhiu0}
Assume that $\rho$ is induced by a Schwartz probability density function $f : V \to [0,\infty)$, and let $F$ be the corresponding periodized density \eqref{0irhdvp5}. Then, for every decoder $\mathcal{D} \in \mathrm{Dec}_\Lambda$, the probability of no logical error using $\mathcal{D}$ is given by
\[
\pi_*\rho(\im \mathcal{D})
=
\int_{V/\Lambda^\perp} F(\mathcal{D}(y))dy,
\]
where $dy$ is the volume form on $V/\Lambda^\perp$ induced by the hermitian metric. In particular, if $\mathcal{M} \in \mathrm{Dec}_\Lambda$ satisfies
\[
F(\mathcal{M}(y))
=
\max\{F(x) : x \in \phi^{-1}(y)\}
\]
for almost all $y \in V/\Lambda^\perp$, then $\mathcal{M}$ is a maximum likelihood decoder.
\end{proposition}

\begin{proof}
Let $dv$ be the volume form on $V$ induced by the hermitian metric, and let $dx$ be the induced volume form on $V/\Lambda$. Since $f$ is a probability density for $\rho$, the pushforward measure $\pi_*\rho$ has density $F$ with respect to $dx$. Indeed, for every bounded measurable function $g : V/\Lambda \to \R$, choosing a measurable fundamental domain $Q \subset V$ for $\Lambda$, we have
\[
\int_{V/\Lambda} g d(\pi_*\rho)
=
\int_V g(\pi(v))f(v)dv
=
\sum_{\lambda \in \Lambda}\int_Q g(\pi(v+\lambda))f(v+\lambda)dv
=
\int_{V/\Lambda} g(x)F(x)dx.
\]
Thus $d(\pi_*\rho)=Fdx$.

For a decoder $\mathcal{D}$, the probability of no logical error is
\[
\pi_*\rho(\im \mathcal{D})
=
\int_{\im \mathcal{D}} F(x)dx.
\]
Since $\mathcal{D}$ is a section of $\phi : V/\Lambda \to V/\Lambda^\perp$, the restriction
\[
\phi : \im \mathcal{D} \too V/\Lambda^\perp
\]
is a measurable bijection with inverse $\mathcal{D}$, and the quotient volume forms are compatible with this covering map. Hence
\[
\int_{\im \mathcal{D}} F(x)dx
=
\int_{V/\Lambda^\perp} F(\mathcal{D}(y))dy.
\]
The last statement follows immediately.
\end{proof}

Let $\mathcal{M}$ be a maximum likelihood decoder obtained by maximizing $F$ on the fibres of $\phi$, as in Proposition \ref{rq2fhiu0}. The set of errors corrected successfully by $\mathcal{M}$ is
\[
G
\coloneqq
\{v \in V : F(\pi(v)) = \max_{x \in \phi^{-1}(\phi(\pi(v)))} F(x)\}.
\]
Thus
\[
\mathscr{R}_\rho(\Lambda) = \rho(G)
\quad\text{and}\quad
\mathscr{F}_\rho(\Lambda) = \rho(V\setminus G).
\]
Equivalently, since the points in the fibre of $\phi$ through $\pi(v)$ are represented by $\pi(v+\mu)$ with $\mu \in \Lambda^\perp$, we may write
\[
G
=
\bigcap_{[\mu] \in \Lambda^\perp/\Lambda}G_\mu,
\]
where
\[
G_\mu
\coloneqq
\{v \in V : F(v) \ge F(v+\mu)\}.
\]

\subsection{Small variance limit}

We now specialize to Gaussian displacement errors. Let $\rho_\sigma$ be the probability measure on $V$ induced by the centered isotropic Gaussian density of variance $\sigma^2$, namely $d\rho_\sigma = f_\sigma dv$, where
\[
f_\sigma : V \too \R, \quad f_\sigma(v) = \frac{1}{(2\pi \sigma^2)^n} \exp(-\frac{|v|^2}{2\sigma^2}),
\]
and $n = \tfrac{1}{2}\dim_\R V$.
We write
\[
\mathscr{R}_\sigma(\Lambda)
\coloneqq
\mathscr{R}_{\rho_\sigma}(\Lambda)
\quad\text{and}\quad
\mathscr{F}_\sigma(\Lambda)
\coloneqq
\mathscr{F}_{\rho_\sigma}(\Lambda).
\]
The goal of this subsection is to compute the asymptotic behaviour of $\mathscr{F}_\sigma(\Lambda)$ as $\sigma \to 0$.

Define
\[
\ell(\Lambda)
\coloneqq
\min\{|\mu| : \mu \in \Lambda^\perp\setminus\Lambda\}
\]
and
\[
N(\Lambda)
\coloneqq
|\{\mu \in \Lambda^\perp\setminus\Lambda : |\mu| = \ell(\Lambda)\}|.
\]
The number $\ell(\Lambda)$ is the length of the shortest displacement that can produce a non-trivial logical Pauli gate. The goal of this subsection is to prove the following result.

\begin{theorem}\label{clp62a30}
For every GKP code $(\Lambda,\nu)$ in $V$, we have
\[
\mathscr{F}_\sigma(\Lambda)
\sim
\frac{2N(\Lambda)\sigma}{\ell(\Lambda)\sqrt{2\pi}}
\exp(-\frac{\ell(\Lambda)^2}{8\sigma^2})
\quad
\text{as } \sigma \to 0.
\]
\end{theorem}

Thus, in the small-noise limit, the optimal probability of a logical error is governed to first order by the shortest vectors in $\Lambda^\perp\setminus\Lambda$.

An immediate consequence is that, for sufficiently small noise, codes with larger $\ell(\Lambda)$ have smaller logical error probability.

\begin{corollary}
Let $(\Lambda_1,\nu_1)$ and $(\Lambda_2,\nu_2)$ be two GKP codes in $V$. If $\ell(\Lambda_1) > \ell(\Lambda_2)$, or if $\ell(\Lambda_1) = \ell(\Lambda_2)$ and $N(\Lambda_1) < N(\Lambda_2)$, then there exists $\sigma_0 > 0$ such that
\[
\mathscr{R}_\sigma(\Lambda_1) > \mathscr{R}_\sigma(\Lambda_2)
\]
for all $0 < \sigma < \sigma_0$.
\end{corollary}

\begin{remark}
If $\Lambda$ has type $(d,d,\ldots,d)$ for some $d > 1$, then
\[
\ell(\Lambda) = \frac{1}{d}\lambda_1(\Lambda),
\]
where
\[
\lambda_1(\Lambda) \coloneqq \min\{|\lambda| : \lambda \in \Lambda\setminus\{0\}\}.
\]
Equivalently, $\lambda_1(\Lambda)$ is the systole of the torus $V/\Lambda$, i.e.\ the length of the shortest closed geodesic. It is also twice the packing radius of the lattice $\Lambda$. Thus, in this case, maximizing $\ell(\Lambda)$ is equivalent to maximizing the sphere-packing radius of $\Lambda$.
\end{remark}

The proof of Theorem \ref{clp62a30} occupies the rest of this subsection. The main point is to compare the maximum likelihood decoding region with the Voronoi cell of $\Lambda^\perp$ and to show that, in the small-variance limit, only the facets corresponding to shortest vectors in $\Lambda^\perp\setminus\Lambda$ contribute to first order.

We begin with a lattice-sum bound that will be used repeatedly.

\begin{lemma}\label{6vgpl6zb}
Let $\Lambda \subset \R^m$ be a lattice and let $S \subset \Lambda$ be any subset. Let
\[
\ell \coloneqq \min\{|\lambda| : \lambda \in S\}.
\]
Then for every $a_0 > 0$ there exists $c > 0$ such that
\[
\sum_{\lambda \in S} e^{-a|\lambda|^2} \le ce^{-a\ell^2}
\]
for all $a > a_0$.
\end{lemma}

\begin{proof}
We write
\[
\sum_{\lambda \in S} e^{-a|\lambda|^2}
=
e^{-a\ell^2}
\sum_{\lambda \in S} e^{-a(|\lambda|^2-\ell^2)}.
\]
Since $|\lambda|^2-\ell^2 \ge 0$ for all $\lambda \in S$, the second factor is bounded above by
\[
\sum_{\lambda \in S} e^{-a_0(|\lambda|^2-\ell^2)}
=
e^{a_0\ell^2}\sum_{\lambda \in S} e^{-a_0|\lambda|^2}.
\]
The last sum is finite because $S \subset \Lambda$, so the result follows.
\end{proof}

\begin{lemma}\label{y3bic8v7}
Let $\mu \in V\setminus\{0\}$ and let $r:(0,\infty) \to \R$ be bounded below by a positive constant. Then
\[
\rho_\sigma(\{v \in V : v\cdot\mu \ge r(\sigma)\})
\le
\frac{\sigma|\mu|}{r(\sigma)\sqrt{2\pi}}
\exp(-\frac{r(\sigma)^2}{2\sigma^2|\mu|^2})
\]
for all $\sigma > 0$. Moreover,
\[
\rho_\sigma(\{v \in V : v\cdot\mu \ge r(\sigma)\})
\sim
\frac{\sigma|\mu|}{r(\sigma)\sqrt{2\pi}}
\exp(-\frac{r(\sigma)^2}{2\sigma^2|\mu|^2})
\]
as $\sigma \to 0$.
\end{lemma}

\begin{proof}
Under $\rho_\sigma$, the random variable
\[
Z = \frac{v\cdot\mu}{\sigma|\mu|}
\]
is a centered one-dimensional Gaussian of variance $1$. Set
\[
u(\sigma) = \frac{r(\sigma)}{\sigma|\mu|}.
\]
Since $r(\sigma)$ is bounded below by a positive constant, we have $u(\sigma)\to\infty$ as $\sigma\to 0$. Moreover,
\[
\rho_\sigma(\{v \in V : v\cdot\mu \ge r(\sigma)\})
=
\frac{1}{\sqrt{2\pi}}\int_{u(\sigma)}^\infty e^{-t^2/2}dt.
\]
The result follows from the standard Gaussian tail estimate
\[
\int_u^\infty e^{-t^2/2}dt \le \frac{1}{u}e^{-u^2/2}
\]
and the asymptotic
\[
\int_u^\infty e^{-t^2/2}dt \sim \frac{1}{u}e^{-u^2/2}
\]
as $u\to\infty$.
\end{proof}

For the rest of this subsection, let $\FDA \subset V$ denote the Voronoi cell of $\Lambda^\perp$ at the origin, i.e.
\[
\FDA
=
\{v \in V : |v| \le |v-\mu| \text{ for all } \mu \in \Lambda^\perp\}.
\]
Thus $\FDA$ is a fundamental domain for $\Lambda^\perp$, up to a set of measure zero. For $\mu \in V$, let
\[
H_\mu
\coloneqq
\{v \in V : v\cdot\mu \ge \tfrac{1}{2}|\mu|^2\} = \{v \in V : |v| \ge |v - \mu| \}.
\]
This is the half-space bounded by the perpendicular bisector of $0$ and $\mu$, on the side closer to $\mu$.

Let
\[
F_\sigma(v)
=
\sum_{\lambda \in \Lambda} f_\sigma(v+\lambda)
\]
be the periodized Gaussian density associated with $\Lambda$.

\begin{proposition}\label{0i0s2i2p}
We have
\begin{equation}\label{qcu8bldg}
\mathscr{F}_\sigma(\Lambda)
\le
\frac{2\sigma}{\sqrt{2\pi}}
\sum_{\mu \in \Lambda^\perp\setminus\Lambda}
\frac{1}{|\mu|}
\exp(-\frac{|\mu|^2}{8\sigma^2})
\end{equation}
for all $\sigma > 0$. In particular,
\[
\mathscr{F}_\sigma(\Lambda)
\le
\frac{2N(\Lambda)\sigma}{\ell(\Lambda)\sqrt{2\pi}}
\exp(-\frac{\ell(\Lambda)^2}{8\sigma^2})(1+o(1)).
\]
\end{proposition}

\begin{proof}
By Proposition \ref{rq2fhiu0}, the maximum likelihood decoder has success probability
\[
\mathscr{R}_\sigma(\Lambda)
=
\int_{\FDA}
\max_{[\eta]\in\Lambda^\perp/\Lambda}
F_\sigma(v+\eta)dv.
\]
Here we use $\FDA$ as a fundamental domain for $V/\Lambda^\perp$. Therefore
\[
\mathscr{R}_\sigma(\Lambda)
\ge
\int_{\FDA}F_\sigma(v)dv.
\]
Using the definition of $F_\sigma$, we get
\[
\int_{\FDA}F_\sigma(v)dv
=
\sum_{\lambda \in \Lambda}
\int_{\FDA} f_\sigma(v+\lambda)dv
=
\rho_\sigma\left(\bigcup_{\lambda \in \Lambda}(\FDA+\lambda)\right).
\]
Hence
\[
\mathscr{F}_\sigma(\Lambda)
\le
\rho_\sigma\left(V\setminus \bigcup_{\lambda \in \Lambda}(\FDA+\lambda)\right).
\]
Since $V$ is tiled by the translates $\FDA+\mu$ with $\mu \in \Lambda^\perp$, up to a set of measure zero, we have
\[
V\setminus \bigcup_{\lambda \in \Lambda}(\FDA+\lambda)
\subset
\bigcup_{\mu \in \Lambda^\perp\setminus\Lambda}(\FDA+\mu).
\]
Moreover, for every $\mu \in \Lambda^\perp\setminus\Lambda$,
\[
\FDA+\mu \subset H_\mu,
\]
because each point of $\FDA+\mu$ is at least as close to $\mu$ as it is to $0$. It follows that
\[
\mathscr{F}_\sigma(\Lambda)
\le
\sum_{\mu \in \Lambda^\perp\setminus\Lambda}
\rho_\sigma(H_\mu).
\]
Lemma \ref{y3bic8v7}, applied with $r(\sigma)=|\mu|^2/2$, gives
\[
\rho_\sigma(H_\mu)
\le
\frac{2\sigma}{|\mu|\sqrt{2\pi}}
\exp(-\frac{|\mu|^2}{8\sigma^2}).
\]
This proves \eqref{qcu8bldg}.

It remains to extract the first-order term. Let
\[
\ell_2
=
\min\{|\mu| : \mu \in \Lambda^\perp\setminus\Lambda,\ |\mu|>\ell(\Lambda)\}.
\]
By Lemma \ref{6vgpl6zb}, we have
\[
\sum_{\substack{\mu \in \Lambda^\perp\setminus\Lambda\\|\mu|>\ell(\Lambda)}}
\frac{1}{|\mu|}
\exp(-\frac{|\mu|^2}{8\sigma^2})
=
O\left(\exp(-\frac{\ell_2^2}{8\sigma^2})\right).
\]
Since $\ell_2>\ell(\Lambda)$, this is
\[
o\left(\exp(-\frac{\ell(\Lambda)^2}{8\sigma^2})\right).
\]
The terms with $|\mu|=\ell(\Lambda)$ contribute
\[
\frac{N(\Lambda)}{\ell(\Lambda)}
\exp(-\frac{\ell(\Lambda)^2}{8\sigma^2}),
\]
and the stated asymptotic upper bound follows.
\end{proof}

Let $\FDB \subset V$ be the Voronoi cell of $\Lambda$ at the origin, and let $\operatorname{Int}(\FDB)$ denote its interior. We will use the following estimate to compare the periodized Gaussian density $F_\sigma$ with the original Gaussian density $f_\sigma$ on compact subsets of $\operatorname{Int}(\FDB)$.

\begin{lemma}\label{m4fgiqyn}
Define $\delta_\sigma(v)$ by
\[
F_\sigma(v) = f_\sigma(v)(1+\delta_\sigma(v)).
\]
For every compact set $K \subset \operatorname{Int}(\FDB)$, there exist constants $\sigma_0, c, a>0$ such that, for all $0<\sigma<\sigma_0$ and all $v \in K$, we have
\[
|\delta_\sigma(v)|\le ce^{-a/\sigma^2}.
\]
\end{lemma}

\begin{proof}
By definition,
\[
\delta_\sigma(v)
=
\sum_{\lambda\in\Lambda\setminus\{0\}}
\exp(-\frac{|v+\lambda|^2-|v|^2}{2\sigma^2}).
\]
Since $K \subset \operatorname{Int}(\FDB)$, we have $|v+\lambda|^2-|v|^2>0$ for all $v\in K$ and all $\lambda\in\Lambda\setminus\{0\}$. Moreover,
\[
|v+\lambda|^2-|v|^2=|\lambda|^2+2v\cdot\lambda.
\]
Since $K$ is compact, there exists $a>0$ such that
\[
|v+\lambda|^2-|v|^2 \ge a|\lambda|^2
\]
for all $v\in K$ and all $\lambda\in\Lambda\setminus\{0\}$. It follows that
\[
|\delta_\sigma(v)|
\le
\sum_{\lambda\in\Lambda\setminus\{0\}}
\exp(-\frac{a|\lambda|^2}{2\sigma^2}),
\]
and the result follows from Lemma \ref{6vgpl6zb}.
\end{proof}

\begin{proof}[Proof of Theorem \ref{clp62a30}]
Set $\ell \coloneqq \ell(\Lambda)$ and
\[
S \coloneqq \{\eta \in \Lambda^\perp \setminus \Lambda : |\eta| = \ell\},
\]
so that $N(\Lambda) = |S|$. Proposition \ref{0i0s2i2p} gives the upper bound. It remains to prove that
\begin{equation}\label{ahph15dd}
\mathscr{F}_\sigma(\Lambda) \ge \frac{2|S|\sigma}{\ell\sqrt{2\pi}}\exp(-\frac{\ell^2}{8\sigma^2})(1+o(1)).
\end{equation}

Let
\[
E_\mu \coloneqq \{v \in V : F_\sigma(v-\mu)>F_\sigma(v)\}.
\]
Then the error region is
\[
E \coloneqq V\setminus G
=
\bigcup_{[\mu]\in\Lambda^\perp/\Lambda}E_\mu,
\]
and hence
\begin{equation}\label{d5erlgm5}
\mathscr{F}_\sigma(\Lambda)
=
\rho_\sigma(E)
=
\rho_\sigma\left(\bigcup_{[\mu]\in\Lambda^\perp/\Lambda}E_\mu\right).
\end{equation}

We first claim that $S \subset \FDB$. Indeed, if $\mu \notin \FDB$, then there exists $\eta \in \Lambda$ such that $|\mu|>|\mu-\eta|$. But then $\mu-\eta \in \Lambda^\perp\setminus\Lambda$ and $|\mu-\eta|<\ell$, contradicting the definition of $\ell$.

It follows that, for each $\mu\in S$, we have
\[
\mu/2 \in \operatorname{Int}(\FDB)
\quad\text{and}\quad
-\mu/2 \in \operatorname{Int}(\FDB).
\]
Since $S$ is finite, we may choose pairwise disjoint closed balls $B_\mu$ centered at $\mu/2$ such that
\[
B_\mu\subset \operatorname{Int}(\FDB),
\qquad
B_\mu-\mu\subset \operatorname{Int}(\FDB)
\]
for all $\mu\in S$.

We will show that there exists a function $\e(\sigma)=o(\sigma^2)$ such that, for all $\mu\in S$,
\begin{equation}\label{o1vdj4c5}
A_\mu(\sigma)
\coloneqq
\{v\in B_\mu : v\cdot\mu \ge \tfrac{\ell^2}{2}+\e(\sigma)\}
\subset E_\mu
\end{equation}
and
\begin{equation}\label{fa6q8y14}
\rho_\sigma(A_\mu(\sigma))
\sim
\frac{2\sigma}{\ell\sqrt{2\pi}}
\exp(-\frac{\ell^2}{8\sigma^2})
\quad \text{as } \sigma \to 0.
\end{equation}
Since the balls $B_\mu$ are pairwise disjoint, \eqref{d5erlgm5}, \eqref{o1vdj4c5}, and \eqref{fa6q8y14} imply \eqref{ahph15dd}.

We first prove the inclusion \eqref{o1vdj4c5}. The set
\[
\bigcup_{\mu \in S}(B_\mu \cup (B_\mu-\mu))
\]
is a compact subset of $\operatorname{Int}(\FDB)$. By Lemma \ref{m4fgiqyn}, there exist constants $b,c>0$ such that, for all $\mu\in S$ and all $v\in B_\mu$,
\[
F_\sigma(v)
=
\frac{1}{(2\pi\sigma^2)^n}
\exp(-\frac{|v|^2}{2\sigma^2})
(1+\delta_{1,\sigma}(v)),
\]
and
\[
F_\sigma(v-\mu)
=
\frac{1}{(2\pi\sigma^2)^n}
\exp(-\frac{|v-\mu|^2}{2\sigma^2})
(1+\delta_{2,\sigma}(v)),
\]
with
\[
|\delta_{i,\sigma}(v)| \le ce^{-b/\sigma^2}.
\]
Thus, for $v\in B_\mu$, we have
\[
v \in E_\mu
\quad\Longleftrightarrow\quad
\exp(\frac{|v|^2-|v-\mu|^2}{2\sigma^2})
>
\frac{1+\delta_{1,\sigma}(v)}{1+\delta_{2,\sigma}(v)}.
\]
Since
\[
|v|^2-|v-\mu|^2 = 2v\cdot\mu-\ell^2,
\]
this is equivalent to
\[
v\cdot\mu
>
\frac{\ell^2}{2}
+
\sigma^2\log\frac{1+\delta_{1,\sigma}(v)}{1+\delta_{2,\sigma}(v)}.
\]
For $\sigma$ sufficiently small, we have
\[
\left|\log\frac{1+\delta_{1,\sigma}(v)}{1+\delta_{2,\sigma}(v)}\right|<4ce^{-b/\sigma^2}.
\]
It follows that \eqref{o1vdj4c5} holds with
\[
\e(\sigma) \coloneqq 4c\sigma^2e^{-b/\sigma^2}.
\]
In particular, $\e(\sigma)=o(\sigma^2)$.

It remains to prove \eqref{fa6q8y14}. Let
\[
H_\mu(\sigma)
\coloneqq
\{v \in V : v\cdot\mu \ge \tfrac{\ell^2}{2}+\e(\sigma)\}.
\]
Then $A_\mu(\sigma)=B_\mu\cap H_\mu(\sigma)$. The point of minimal norm in $H_\mu(\sigma)$ is
\[
p_\sigma
=
\left(\frac{\ell^2}{2}+\e(\sigma)\right)\frac{\mu}{\ell^2}.
\]
Since $p_\sigma\to \mu/2$ and $B_\mu$ is a neighbourhood of $\mu/2$, there exists $r_0>\ell/2$ such that, for all sufficiently small $\sigma$,
\[
H_\mu(\sigma)\setminus B_\mu
\subset
\{v\in V : |v|\ge r_0\}.
\]
By the standard Gaussian tail estimate,
\[
\rho_\sigma(H_\mu(\sigma)\setminus B_\mu)
\le
\rho_\sigma(\{v\in V : |v|\ge r_0\})
=
O\left(\frac{1}{\sigma^{2n-2}}\exp(-\frac{r_0^2}{2\sigma^2})\right).
\]
Since $r_0>\ell/2$, this implies
\begin{equation}\label{kzw3oe07}
\rho_\sigma(H_\mu(\sigma)\setminus B_\mu)
=
o\left(\sigma\exp(-\frac{\ell^2}{8\sigma^2})\right).
\end{equation}
On the other hand, Lemma \ref{y3bic8v7} gives
\begin{equation}\label{tonwplmz}
\rho_\sigma(H_\mu(\sigma))
\sim
\frac{\sigma\ell}{(\tfrac{\ell^2}{2}+\e(\sigma))\sqrt{2\pi}}
\exp(-\frac{(\tfrac{\ell}{2}+\frac{\e(\sigma)}{\ell})^2}{2\sigma^2})
\sim
\frac{2\sigma}{\ell\sqrt{2\pi}}
\exp(-\frac{\ell^2}{8\sigma^2})
\end{equation}
as $\sigma\to 0$, where the second equivalence uses $\e(\sigma)=o(\sigma^2)$.
Finally, \eqref{fa6q8y14} follows from \eqref{kzw3oe07} and \eqref{tonwplmz}. This proves the lower bound \eqref{ahph15dd}, and hence the theorem.
\end{proof}

This completes the analysis of optimal decoding for Gaussian displacement noise.

\subsection{Geometric reformulation}

We now reinterpret these notions in the language of complex abelian varieties, thereby recovering the description in Section \ref{fci02gnv}.

The first task is to reinterpret displacement errors $\psi \mto T_v\psi$ geometrically. Let $(\Lambda,\nu)$ be a GKP code in $V$ and let $(X,L)$ be the corresponding polarized complex abelian variety. Recall that identifying $H^0(X,L)$ with the space of theta functions $\Theta_{\Lambda,\nu}$ and then applying the Bargmann transform gives an embedding
\[
\iota_0 : H^0(X,L) \longhook \mathcal{S}'(\R^n).
\]
More generally, for each $v \in V$, we have an identification
\[
H^0(X,t_{\bar v}^*L) \cong \Theta_{\Lambda,\nu e^{2\pi iE(v,\cdot)}}
\]
by \cite[Lemma 2.3.2]{BirkenhakeLange2004ComplexAbelianVarieties}, and hence an embedding
\[
\iota_v : H^0(X,t_{\bar v}^*L) \longhook \mathcal{S}'(\R^n);
\]
see the proof of Proposition \ref{oa1igqfi} for an explicit description.

Now, the spaces $H^0(X,L)$ and $H^0(X,t_{\bar v}^*L)$ are also naturally related by the unitary isomorphism
\begin{equation}\label{vtah97po}
t_{\bar v}^* : H^0(X,L) \too H^0(X,t_{\bar v}^*L), \quad
s \mtoo t_{\bar v}^*s \coloneqq \hat{t}_{\bar v}^{-1}\circ s\circ t_{\bar v},
\end{equation}
where $\hat{t}_{\bar v} : t_{\bar v}^*L \to L$ is the canonical map covering $t_{\bar v}$.

\begin{proposition}\label{oa1igqfi}
For all $v \in V$, the diagram
\[
\begin{tikzcd}
H^0(X,L) \arrow{r}{t_{\bar v}^*} \arrow[hook]{d}{\iota_0} & H^0(X,t_{\bar v}^*L) \arrow[hook]{d}{\iota_v} \\
\mathcal{S}'(\R^n) \arrow{r}{T_v} & \mathcal{S}'(\R^n)
\end{tikzcd}
\]
commutes.
\end{proposition}

\begin{proof}
Recall that $L=L(H,\nu)$ is given by the quotient \eqref{zcsp4rse}. Thus we have an identification $H^0(X,L)\cong \Theta_{\Lambda,\nu}$ which sends a theta function $f\in \Theta_{\Lambda,\nu}$ to the section
\[
s : X \too L, \quad s([u])=[u,f(u)].
\]

Similarly, we identify $H^0(X,t_{\bar v}^*L)$ with $\Theta_{\Lambda,\nu e^{2\pi iE(v,\cdot)}}$ as follows. Using
\[
t_{\bar v}^*L = \{([u],[u+v,z]) \in X\times L : u \in V,\ z \in \C\},
\]
we associate to $f\in \Theta_{\Lambda,\nu e^{2\pi iE(v,\cdot)}}$ the section
\begin{equation}\label{3orscsx1}
s : X \too t_{\bar v}^*L, \quad
s([u]) = ([u],[u+v,e^{\pi(\frac{H(v,v)}{2}+H(u,v))}f(u)]).
\end{equation}
The term $\frac{\pi}{2}H(v,v)$ in the exponential is not needed for the automorphy rule, but it makes the identification unitary with respect to the natural hermitian metric on $H^0(X,t_{\bar v}^*L)$ induced by that of $L$ and the natural inner product on theta functions \eqref{cu7nnb56}.

Let $s\in H^0(X,L)$ and let $f\in \Theta_{\Lambda,\nu}$ be the corresponding theta function. Then
\[
(t_{\bar v}^*s)([u]) = ([u],[u+v,f(u+v)]).
\]
Therefore, if $g\in \Theta_{\Lambda,\nu e^{2\pi iE(v,\cdot)}}$ is the theta function corresponding to $t_{\bar v}^*s$, then
\[
e^{\pi(\frac{H(v,v)}{2}+H(u,v))}g(u)=f(u+v)
\]
for all $u\in V$. Hence the map \eqref{vtah97po} corresponds to
\[
D_v : \Theta_{\Lambda,\nu} \too \Theta_{\Lambda,\nu e^{2\pi iE(v,\cdot)}}, \quad
(D_vf)(u)=e^{-\pi(\frac{H(v,v)}{2}+H(u,v))}f(u+v).
\]
By \eqref{c03f536z}, this is exactly the action of $T_v$ under the Bargmann transform. This proves the commutativity of the diagram.
\end{proof}

For a point $x \in X$, the identification
\[
H^0(X,t_x^*L) \cong \Theta_{\Lambda,\nu e^{2\pi iE(v,\cdot)}}
\]
depends on a choice of lift $v \in V$ of $x$, but changing the lift only changes the resulting identification by a global phase; see \eqref{3orscsx1}. Since global phases do not affect quantum states, the displacement operator $T_v : \im \iota_0 \to \im \iota_v$ is geometrically equivalent to the pullback
\[
t_x^* : H^0(X,L) \too H^0(X,t_x^*L),
\]
independently of the choice of lift $v$.

It follows that, under the identification between holomorphic sections of $L$ and the code space $\code_{\Lambda,\nu}\subset \mathcal{S}'(\R^n)$, a displacement error
\[
\code_{\Lambda,\nu} \ni \psi \mtoo T_v\psi
\]
with $v \in V$ sampled from $\rho$ corresponds geometrically to the pullback map
\[
t_x^* : H^0(X,L) \too H^0(X,t_x^*L),
\]
where $x=[v]\in X$ is sampled from $\pi_*\rho$.

Given $s\in H^0(X,L)$ and an unknown displacement $x\in X$, the syndrome of the error state corresponding to $t_x^*s$ is the image of $x$ in $V/\Lambda^\perp$. Under the identification
\[
V/\Lambda^\perp \cong \widehat{X}=\mathrm{Pic}^0(X),
\]
the quotient map $X=V/\Lambda \to V/\Lambda^\perp$ corresponds to the polarization isogeny
\[
\phi_L : X \too \widehat{X}, \qquad x \mtoo t_x^*L\otimes L^{-1}.
\]
Thus the syndrome of $t_x^*s$ is the point
\[
\phi_L(x)\in \widehat{X}.
\]

Equivalently, the translated line bundle $t_x^*L$ can be written as
\[
t_x^*L \cong \phi_L(x)\otimes L,
\]
where $\phi_L(x)$ is viewed as a degree-zero line bundle on $X$. Thus the possible error spaces are parametrized by $\widehat{X}$ through the vector bundle
\[
\mathscr{M}_L : \bigsqcup_{y \in \widehat{X}} H^0(X,y\otimes L) \too \widehat{X}.
\]
In these terms, the syndrome measurement is precisely the natural projection $\mathscr{M}_L$.

The decoding theory developed above can therefore be expressed intrinsically as follows. Let $\rho$ be a Borel probability measure on $X$. A decoder is a measurable section
\[
\mathcal{D} : \widehat{X} \too X
\]
of the polarization isogeny $\phi_L : X \to \widehat{X}$, and its success probability is $\rho(\im\mathcal{D})$. Thus the robustness is
\[
\mathscr{R}_\rho(X,L)
=
\sup_{\mathcal{D}} \rho(\im\mathcal{D}),
\]
where the supremum is taken over all measurable sections of $\phi_L$. This agrees with Definition \ref{aszc1bm9} when $\rho$ is the pushforward to $X=V/\Lambda$ of a probability measure on $V$.

This formulation shows in particular that $\mathscr{R}_\rho(X,L)$ depends on $L$ only through the polarization isogeny $\phi_L$. Hence it depends only on $X$ and the polarization class $c_1(L)$, and not on the specific representative $L$ of that class.

We now specialize to Gaussian displacement noise. Let $\pi : V \to X$ be the universal cover, equipped with the hermitian metric induced by the polarization. For $\sigma>0$, let $\widetilde{\rho}_{\sigma,(X,L)}$ be the centered Gaussian measure of variance $\sigma^2$ on $V$ with respect to this metric, and set
\[
\rho_{\sigma,(X,L)}
\coloneqq
\pi_*\widetilde{\rho}_{\sigma,(X,L)}.
\]
This construction is compatible with isomorphisms of polarized abelian varieties. Indeed, if
\[
f : (X,L) \too (Y,M)
\]
is an isomorphism of polarized abelian varieties and $\widetilde{f} : V_X \to V_Y$ is a lift to the universal covers, then $\widetilde{f}$ preserves the corresponding hermitian metrics. Hence
\[
(\widetilde{f})_*\widetilde{\rho}_{\sigma,(X,L)}
=
\widetilde{\rho}_{\sigma,(Y,M)},
\]
and therefore
\[
f_*\rho_{\sigma,(X,L)}
=
\rho_{\sigma,(Y,M)}.
\]
It follows that, for every polarization type $D$ and every $\sigma>0$, Gaussian displacement noise defines a robustness function
\[
\mathscr{R}_\sigma : \mathcal{A}_D \too [0,1],
\qquad
[(X,L)] \mtoo \mathscr{R}_{\rho_{\sigma,(X,L)}}(X,L).
\]
Thus Problem \ref{86b8npxy} asks for the optimal value of this function on $\mathcal{A}_D$ and for the geometry of the locus where this value is attained.

Finally, the lattice invariant appearing in Theorem \ref{clp62a30} has the following geometric interpretation. Under the identification
\[
K(L)=\ker\phi_L \cong \Lambda^\perp/\Lambda,
\]
we have
\[
\ell(\Lambda)
=
\min_{0\ne x\in K(L)} d_X(0,x),
\]
where $d_X$ is the distance induced by the flat metric on $X$. Thus $\ell(\Lambda)$ is the length of the shortest geodesic segment from the origin to a nonzero point of the polarization kernel. Similarly, $N(\Lambda)$ is the number of points of $K(L)$ at this minimal distance. In this sense, the small-noise asymptotic is governed by a relative systolic invariant of the polarized abelian variety. Hence Problem \ref{j66sypn9} is the leading-order version of Problem \ref{86b8npxy} in the small-variance limit.

\section{Concatenation as isogenies}\label{cu43tyv7}

We finish by discussing concatenation. In quantum error correction, concatenation means taking a code space and imposing additional stabilizer conditions inside it. For GKP codes, this operation has a simple lattice-theoretic meaning: one passes from the original lattice $\Lambda$ to an intermediate symplectically integral lattice between $\Lambda$ and $\Lambda^\perp$. Geometrically, this is the same as passing to an isogenous abelian variety.

We first recall the finite-dimensional stabilizer formalism in the notation of Section \ref{fospnkpi}. Let $D=\diag(d_1,\ldots,d_n)$, with $d_i\ge 2$, and consider the qudit system
\[
\C^{\otimes D}=\C^{d_1}\otimes\cdots\otimes\C^{d_n}
\]
with Pauli group $\Pauli_D\subset \U(\C^{\otimes D})$.

A \defn{stabilizer group} is a subgroup $S\subset \Pauli_D$ whose simultaneous $+1$ eigenspace
\[
(\C^{\otimes D})_S
\coloneqq
\{\psi \in \C^{\otimes D} : s\psi=\psi \text{ for all } s\in S\}
\]
is nonzero. This eigenspace is called the \defn{stabilizer code} associated with $S$. By \eqref{af95ua00}, stabilizer groups are abelian.

The same definition applies in the basis-independent setting of Section \ref{hcdib98d}. Let $\H$ be a Heisenberg group over a finite symplectic abelian group $K$, let $W$ be its Stone--von Neumann representation, and let $\P\subset\H$ be the corresponding abstract Pauli group. A stabilizer group is a subgroup $S\subset \P$ such that
\[
W_S
\coloneqq
\{\psi\in W : s\psi=\psi \text{ for all } s\in S\}
\]
is nonzero.

In particular, if $(\Lambda,\nu)$ is a GKP code, then we may apply this definition to the Pauli group $\Pauli_{\Lambda,\nu}$ acting on the code space $\code_{\Lambda,\nu}$. Given a stabilizer group $S\subset \Pauli_{\Lambda,\nu}$, the \defn{concatenation} of $(\Lambda,\nu)$ with $S$ is the subspace
\[
(\code_{\Lambda,\nu})_S
\coloneqq
\{\psi\in \code_{\Lambda,\nu} : s\psi=\psi \text{ for all } s\in S\}
\subset \code_{\Lambda,\nu}.
\]
The next result shows that this concatenated code space is itself the code space of a GKP code.

\begin{proposition}\label{ivce7fh0}
For every GKP code $(\Lambda,\nu)$ and every stabilizer group $S\subset \Pauli_{\Lambda,\nu}$, there is a canonical GKP code $(\Lambda_S,\nu_S)$ such that
\[
(\code_{\Lambda,\nu})_S=\code_{\Lambda_S,\nu_S}.
\]
Explicitly,
\[
\Lambda_S \coloneqq \{\mu \in \Lambda^\perp : \mu+\Lambda \in \pi(S)\},
\]
where $\pi : \Pauli_{\Lambda,\nu}\to \Lambda^\perp/\Lambda$ is the canonical projection. The semicharacter $\nu_S$ is uniquely determined by the condition that
\[
[\mu,\nu_S(\mu)^{-1}]\in S
\]
for all $\mu\in\Lambda_S$.
\end{proposition}

\begin{proof}
Since $S$ is abelian, the image $\pi(S)\subset \Lambda^\perp/\Lambda$ is isotropic with respect to the symplectic form on $\Lambda^\perp/\Lambda$. Thus
\[
E(\mu,\kappa)\in\Z
\]
whenever $\mu+\Lambda$ and $\kappa+\Lambda$ lie in $\pi(S)$. It follows that $\Lambda_S$ is a symplectically integral lattice.

Let $\mu\in\Lambda_S$. By definition, there exists $s\in S$ such that $\pi(s)=\mu+\Lambda$. Writing $s=[\kappa,\alpha]$ with $\kappa\in\Lambda^\perp$ and $\alpha\in\U(1)$, we have $\kappa+\Lambda=\mu+\Lambda$. Replacing $\kappa$ by $\mu$ in the same class, we may therefore write $s=[\mu,\alpha]$.

We claim that such an $\alpha$ is unique. Indeed, if $[\mu,\alpha]$ and $[\mu,\beta]$ both lie in $S$, then $[0,\alpha\beta^{-1}]\in S$. Since $(\code_{\Lambda,\nu})_S$ is nonzero, this scalar must act trivially on a nonzero vector, and hence $\alpha=\beta$. This defines a map
\[
\nu_S : \Lambda_S \too \U(1)
\]
by the condition that $[\mu,\nu_S(\mu)^{-1}]\in S$.

Since $S$ is a subgroup, the map $\nu_S$ satisfies the semicharacter relation for the lattice $\Lambda_S$. Thus $(\Lambda_S,\nu_S)$ is a GKP code. Finally, a vector $\psi\in \code_{\Lambda,\nu}$ is fixed by every element of $S$ if and only if it satisfies the stabilizer equations defining $\code_{\Lambda_S,\nu_S}$. Hence $(\code_{\Lambda,\nu})_S=\code_{\Lambda_S,\nu_S}$.
\end{proof}

Let $(\Lambda,\nu)$ be a GKP code and let $S\subset \Pauli_{\Lambda,\nu}$ be a stabilizer group. By Proposition \ref{ivce7fh0}, the concatenation of $(\Lambda,\nu)$ with $S$ is again a GKP code $(\Lambda_S,\nu_S)$, with
\[
\Lambda\subset \Lambda_S\subset \Lambda^\perp
\]
and
\[
\nu_S|_\Lambda=\nu.
\]
Thus concatenation amounts to enlarging the stabilizer lattice, together with the compatible extension of the semicharacter.

Conversely, suppose that $(\widetilde{\Lambda},\widetilde{\nu})$ is a GKP code such that
\[
\Lambda\subset\widetilde{\Lambda}
\quad\text{and}\quad
\widetilde{\nu}|_\Lambda=\nu.
\]
Since $\widetilde{\Lambda}$ is symplectically integral and contains $\Lambda$, we have
\[
\widetilde{\Lambda}\subset \Lambda^\perp.
\]
Let $m=\exp(\Lambda^\perp/\Lambda)$. For every $\mu\in\widetilde{\Lambda}$, we have $m\mu\in\Lambda$, and hence
\[
\widetilde{\nu}(\mu)^{2m}
=
\widetilde{\nu}(m\mu)^2
=
\nu(m\mu)^2.
\]
By Lemma \ref{g8ok30mq}, this implies that
\[
S_{\widetilde{\Lambda}}
\coloneqq
\{[\mu,\widetilde{\nu}(\mu)^{-1}]\in \Pauli_{\Lambda,\nu} : \mu\in\widetilde{\Lambda}\}
\]
is a subgroup of $\Pauli_{\Lambda,\nu}$. It is abelian because $\widetilde{\Lambda}$ is symplectically integral. Moreover, its simultaneous $+1$ eigenspace is precisely $\code_{\widetilde{\Lambda},\widetilde{\nu}}$, so it is a stabilizer group, and Proposition \ref{ivce7fh0} recovers
\[
\widetilde{\Lambda}=\Lambda_{S_{\widetilde{\Lambda}}}
\quad\text{and}\quad
\widetilde{\nu}=\nu_{S_{\widetilde{\Lambda}}}.
\]

Therefore, concatenating a GKP code $(\Lambda,\nu)$ with a stabilizer code is equivalent to passing to a GKP code $(\widetilde{\Lambda},\widetilde{\nu})$ extending $(\Lambda,\nu)$, in the sense that
\[
\Lambda\subset\widetilde{\Lambda}\subset\Lambda^\perp
\quad\text{and}\quad
\widetilde{\nu}|_\Lambda=\nu.
\]

We now translate this statement into the language of abelian varieties. Let
\[
X=V/\Lambda
\quad\text{and}\quad
\widetilde{X}=V/\widetilde{\Lambda}.
\]
The inclusion $\Lambda\subset\widetilde{\Lambda}$ induces a holomorphic homomorphism
\[
f : X=V/\Lambda \too \widetilde{X}=V/\widetilde{\Lambda}
\]
with finite kernel $\widetilde{\Lambda}/\Lambda$. Thus $f$ is an isogeny. Moreover, if
\[
L=L(H,\nu)
\quad\text{and}\quad
\widetilde{L}=L(H,\widetilde{\nu}),
\]
then the condition $\widetilde{\nu}|_\Lambda=\nu$ gives
\[
f^*\widetilde{L}=L.
\]
Thus concatenation corresponds geometrically to passing from $(X,L)$ to an isogenous polarized abelian variety $(\widetilde{X},\widetilde{L})$.

Conversely, let $(X,L)$ and $(Y,M)$ be polarized complex abelian varieties and let
\[
f : X \too Y
\]
be an isogeny such that $f^*M=L$. After identifying their universal covers, we may write
\[
X=V/\Lambda
\quad\text{and}\quad
Y=V/\widetilde{\Lambda}
\]
with $\Lambda\subset\widetilde{\Lambda}$. Writing $L=L(H,\nu)$ and $M=L(H,\widetilde{\nu})$, the equality $f^*M=L$ gives
\[
\widetilde{\nu}|_\Lambda=\nu.
\]
Thus the isogeny determines an extension of GKP codes, and hence a concatenation by a stabilizer code.

We conclude that, in the GKP setting, concatenation with a stabilizer code is exactly the same operation as passing to a suitable isogeny of polarized complex abelian varieties.

\bibliographystyle{plain}
\bibliography{complex-abelian-varieties-and-qec}

\end{document}